\documentclass[11pt, letterpaper]{article}
\usepackage{amsmath, amsthm, amssymb, amsfonts} % 数学符号和定理环境
\usepackage{mathrsfs} % 花体字母
\usepackage{bm} % 粗体数学符号
\usepackage{graphicx} % 插入图片
\usepackage{enumerate} % 自定义列表
\usepackage{geometry} % 页面设置
\usepackage{setspace} % 行距设置
\usepackage{lmodern} % 现代字体
\usepackage{hyperref} % 超链接
\usepackage{color} % 颜色
\usepackage{xcolor} % 更多颜色选项
\usepackage{url} % URL处理
\usepackage{mathtools}
\usepackage{enumitem}
\usepackage{changepage}
\usepackage{microtype}
\usepackage{authblk}
\usepackage{amsfonts}
\usepackage{mathrsfs,amscd,amssymb,amsthm,amsmath,bm,graphicx,psfrag,subfigure,url,mathtools}
\usepackage{pict2e}
\usepackage{psfrag,amsmath}
\usepackage{tikz}
\usepackage{indentfirst}
\usepackage{hyperref}
\usepackage{bookmark}
\usepackage{enumerate}
\usepackage{latexsym,euscript,epic,eepic,color}
\usepackage{multirow}
\usepackage{multicol}
\usepackage{longtable}
\usepackage{adjustbox}

\usepackage{setspace}
\usepackage{epstopdf}
\allowdisplaybreaks
\usepackage{authblk}
\usepackage{pifont}
\numberwithin{equation}{section}

\usepackage{enumitem}
\usepackage{microtype}

\theoremstyle{plain}

\newtheorem{theorem}{Theorem}[section]
\newtheorem{lemma}[theorem]{Lemma}
\newtheorem{proposition}[theorem]{Proposition}

\newtheorem{claim}{Claim}
\theoremstyle{definition}

\theoremstyle{remark}

\DeclareMathOperator{\spex}{spex}

\newcommand{\cE}{\mathcal E}
\newcommand{\cO}{\mathcal O}
\newcommand{\cD}{\mathcal D}

\newenvironment{wst}
{\setlength{\leftmargini}{1.5\parindent}
 \begin{itemize}
 \setlength{\itemsep}{-1.1mm}}
{\end{itemize}}
\hypersetup{
    colorlinks=true,
    linkcolor=blue,
    citecolor=red,
    urlcolor=magenta,
}

\newcommand{\keywords}[1]{%
  \par\vspace{6pt}\noindent\textbf{Keywords: }#1\par
}

\newcommand{\MSC}[2][2020]{%
  \par\vspace{3pt}\noindent\textbf{MSC(#1): }#2\par
}
\title{Spectral extremal graphs for $W_5$-free graphs with odd size}
\author[1]{Jing Gao}
\author[2]{Xianya Geng}
\author[,1]{Shuchao Li\thanks{Corresponding author. \\
\hspace*{2em}E-mail address: gjing1270@163.com (J. Gao), gengxianya@sina.com (X. Geng), lscmath@ccnu.edu.cn (S. Li)}}
\affil[1]{School of Mathematics and Statistics, and Hubei Key Lab--Math. Sci.,\linebreak Central China Normal University, Wuhan 430079, China}
\affil[2]{School of Mathematics and Big Data,\linebreak Anhui University of Science and Technology, Huainan 232001, China}
\date{}

\allowdisplaybreaks
\begin{document}
\baselineskip=0.23in
% 标题页

\maketitle

\begin{abstract}
For a fixed integer $k\ge 2$, let $W_{2k+1}=K_1\vee C_{2k}$ be an odd wheel graph. The fixed-size spectral extremal problem aims to determine
\[
  \operatorname{spex}(m,W_{2k+1}):=\max\{\rho(G): e(G)=m,\ G \text{ is } W_{2k+1}\text{-free}\},
\]
where $\rho(G)$ denotes the adjacency spectral radius. Based on this problem, Yu, Li, and Peng~\cite{YLP} proposed the following conjecture:  When %$m-\binom{k}{2}$ is divisible by $k$ and
$m$ is large, every $W_{2k+1}$-free graph of size $m$ satisfies
\(
 \rho(G)^2-(k-1)\rho(G)\le m-\binom{k}{2}
\)
with equality precisely for $K_k\vee qK_1$, where $q=\frac{m-\binom{k}2}{k}$. For nonzero residue class, Yu, Zhang, and Zhang~\cite{YZZ} proposed the following conjecture: Let $r$ be a nonzero remainder when $m-\binom{k}{2}$ is divided by $k$ and $m$ is large. Then $S_{k,m}$ is the unique graph among $W_{2k+1}$-free graphs of size $m$ having maximum spectral radius, where $S_{k,m}$ is obtained from $K_k\vee qK_1$ by adding a vertex $z$ and joining it to exactly $r$ vertices of the $K_k$. Very recently, Fang, Zhai and Zhang~\cite{FangZhaiZhang} confirmed the Yu--Li--Peng conjecture for $k\ge 2$. Chen, Gao and Li~\cite{cgl2026} confirmed Yu-Zhang-Zhang conjecture for $k\ge 3$.   When $k=2$, then $W_5=K_1\vee C_4$. For large odd $m$, determining $\operatorname{spex}(m,W_5)$ is still open. In this paper we address the odd-size problem.  Our result disproved Yu-Zhang-Zhang conjecture for $k= 2$.  In our proof, a universal defect bound shows that only $O(1)$ edges can lie outside the dense core.  Perron localization then reduces this to at most one edge.  A discrete defect inequality forces the complete bipartite crossing and quantizes the two matching deficiencies.  Exact quotient-polynomial comparisons eliminate the remaining cross-edge and odd--odd candidates.
\end{abstract}

\keywords{Dense core; Fixed size; Odd wheel; Spectral defect; Spectral radius}

\MSC{05C50, 05C35}

\section{Introduction}

All graphs are finite, simple and undirected. A graph $G$ is said to be $H$-free if $G$ does not contain $H$ as a (not necessarily induced) subgraph.
%Given a family of graphs $\mathcal{H}$, $G$ is $\mathcal{H}$-free if it is $H$-free for every $H\in\mathcal{H}$.
For a graph \(G\), write \(A(G)\) for its adjacency matrix, \(\rho(G)\) for the spectral radius of \(A(G)\), and $e(G)$ the number of edges of $G$.

A central topic in extremal graph theory is the \textit{Tur\'an type problem}: Determine
\(
\operatorname{ex}(n,F)=\max\bigl\{e(G)\colon |V(G)|=n,\, G\text{ is }F\text{-free}\bigr\},
\)
the maximum number of edges in an $n$-vertex $F$-free graph. For a fixed graph $F$, put
\[
 \spex(m,F):=\max\{\rho(G):|E(G)|=m,\ G\text{ is }F\text{-free}\}.
\]
This is the fixed-size version of the spectral Tur\'an problem and goes
back to Brualdi and Hoffman~\cite{BrualdiHoffman}.  Exact fixed-size
results are known for several forbidden cycles, complete bipartite
graphs, fans, friendship graphs, theta graphs, and related color-critical
graphs; see, for example,
\cite{JoyentanujYamini,LiZhaiShu,LiZhaoZou,LiuLiLiYu,ZhaiLinShu,ZhangWang}.
The leading-order structure is supplied by edge-spectral stability
\cite{LiLiuZhangStability}, while exact results require a discrete
second-order analysis.

An odd wheel is $W_{2k+1}=K_1\vee C_{2k}$.  Thus a graph is
$W_{2k+1}$-free precisely when the neighborhood of every vertex induces
no $C_{2k}$.  Fang, Zhai and Zhang~\cite{FangZhaiZhang} recently proved a
sharp fixed-size theorem for every odd wheel.  The exceptional case
$k=2$ reads as follows.

\begin{theorem}[Fang--Zhai--Zhang~\cite{FangZhaiZhang}]
\label{thm:FZZ}
There is an integer $m_0$ such that every $W_5$-free graph $G$ with
$m\ge m_0$ edges satisfies $\rho(G)\leq g^*(m),$ where $g^*(m)=\frac{1+\sqrt{4m+1}}{2}$ is the largest zero of $ x^2-x= m.$ Furthermore, $\rho(G)= g^*(m)$ if and only if
\[
 G\cong K_{n,n}+M_A+M_B,
 \qquad m=n^2+n,
\]
where $n$ is even and $M_A,M_B$ are perfect matchings in the two parts.
\end{theorem}

The bound in Theorem~\ref{thm:FZZ} is attained only at the sparse
sequence $m=n^2+n$.  In particular, it leaves open the exact answer for
every odd $m$.  The odd case is arithmetically more varied than the
zero-defect equality case: three infinite families, as well as the book,
can be extremal.

For odd $m$, define
\[
 B_m:=K_2\vee\overline K_{(m-1)/2},
 \qquad
 \beta_m:=\rho(B_m)=\frac{1+\sqrt{4m-3}}2.
\]
Thus
\begin{equation}
 \beta_m^2-\beta_m=m-1.
\label{eq:book-defect}
\end{equation}
We next describe the other candidates.  If $a,b$ are even, let
$E_{a,b}$ be obtained from $K_{a,b}$ by inserting a perfect matching in
each part.  If $e$ is even and $o$ is odd, let $O_{e,o}$ be obtained from
$K_{e,o}$ by inserting a perfect matching in the $e$-part and a maximum
matching in the $o$-part.  Finally, for even $2\le a\le b$, let
$D_{a,b}$ be obtained from $E_{a,b}$ by deleting one edge of the perfect
matching in the $b$-part.  Put
\begin{align*}
 \cE_m&:=\{E_{a,b}:a\le b,\ a,b\text{ even},\
                 ab+(a+b)/2=m\},\\
 \cO_m&:=\{O_{e,o}:e\text{ even},\ o\text{ odd},\
                 eo+e/2+(o-1)/2=m\},\\
 \cD_m&:=\{D_{a,b}:a\le b,\ a,b\text{ even},\
                 ab+(a+b)/2-1=m\}.
\end{align*}
Equivalently, their parameters are the admissible factor pairs in
\begin{equation}\notag
\begin{aligned}
 (2a+1)(2b+1)&=4m+1 &&(E_{a,b}\in\cE_m),\\
 (2e+1)(2o+1)&=4m+3 &&(O_{e,o}\in\cO_m),\\
 (2a+1)(2b+1)&=4m+5 &&(D_{a,b}\in\cD_m).
\end{aligned}
\end{equation}

When $n$ is even, write $E_n:=E_{n,n}$, and let $H_n$ be obtained from
$E_n$ by attaching a pendant edge to one vertex.  The following is our
main result.

\begin{theorem}
\label{thm:main}
There exists $m_0$ such that the following holds for every odd
$m\ge m_0$.

\begin{enumerate}[label=\textup{(\roman*)}]
 \item If $m=n^2+n+1$ for an even integer $n$, then
 \(
   \spex(m,W_5)=\rho(H_n),
 \)
 and $H_n$ is the unique extremal graph.

 \item Otherwise,
 \begin{equation}
 \spex(m,W_5)=
 \max\bigl(\{\beta_m\}\cup
 \{\rho(H):H\in\cE_m\cup\cO_m\cup\cD_m\}\bigr).
 \label{eq:main-maximum}
 \end{equation}
 The extremal graphs are exactly the members of
 \(
  \{B_m\}\cup\cE_m\cup\cO_m\cup\cD_m
 \)
 whose spectral radius equals the right-hand side of
 \eqref{eq:main-maximum}.
\end{enumerate}
\end{theorem}

%The formula is finite and completely effective: the relevant graphs are
%read off from the three factorizations in \eqref{eq:factorizations}, and
%their spectral radii are roots of polynomials of degree at most three. The
%exceptional pendant radius is the largest root of an explicit quintic;
%see Proposition~\ref{prop:candidate-radii}.

\noindent{\bf Our approach.}\ The proof has four steps.  First, we lower the dense-core threshold in
\cite{FangZhaiZhang} from the universal root
$\frac12(1+\sqrt{4m+1})$ to the book root $\beta_m$.  The gap between
the two normalized thresholds is only $O(m^{-1})$; hence only $O(1)$
edges are lost in the core extraction.  Second, a Perron-coordinate
estimate shows that at most one edge can remain outside the induced
core.  Third, a completion inequality for the \textit{defect}
\[
 \gamma(G):=e(G)-\bigl(\rho(G)^2-\rho(G)\bigr)
\]
forces the crossing graph to be complete and leaves only the three
matching-deficiency patterns in Theorem~\ref{thm:main}.  Finally, exact
quotient-polynomial comparisons remove a deleted cross edge and the
odd--odd parity pattern.  The pendant case is selected by a resolvent
estimate.\vspace{2mm}

\noindent{\bf Organization.}\ \ In the remainder of this section, we give some necessary notation and terminology.  In Section \ref{s2}, we give some preliminaries. In Section~\ref{s3}, we first give some key structure lemmas for characterization of the extremal graph. Then we give the proof of Theorem~\ref{thm:main}. Some concluding remarks are given in the last section.\vspace{2mm}

\noindent{\bf Some notation and definitions.}\ For a vertex $v\in V(G)$, $N_G(v)$ denotes its neighborhood and $d_G(v)=|N_G(v)|$ its degree;
when no confusion arises, we drop the subscript. For two graphs $H_1$ and $H_2$, $H_1\cup H_2$ denotes their vertex-disjoint union,
and $H_1\vee H_2$ denotes their \textit{join}, obtained from the disjoint union $H_1\cup H_2$
by adding all edges between $V(H_1)$ and $V(H_2)$.

For a vertex subset $S\subseteq V(G)$, $G[S]$ denotes the subgraph of $G$ induced by $S$,
and we write $e(S):=e(G[S])$ for the number of edges inside $S$.
For two disjoint vertex subsets $S,T\subseteq V(G)$, $e(S,T)$ denotes the number of edges
with one endpoint in $S$ and the other in $T$, and for $v\in V(G)$ we write $d_R(v):=|N_G(v)\cap R|$.

The \textit{adjacency matrix} $A(G)$ of $G$ is the $n\times n$ $0$-$1$ matrix indexed by $V(G)$ with $A(G)_{uv}=1$ if and only if $uv\in E(G)$.
Since $A(G)$ is real and symmetric, its eigenvalues are real; we denote the largest of them by $\rho(G)$, called the \textit{spectral radius} of $G$.
When $G$ is connected, $A(G)$ is irreducible and nonnegative, so by the Perron-Frobenius theorem $\rho(G)$ is a simple eigenvalue admitting a positive eigenvector $\mathbf{x}=(x_v)_{v\in V(G)}$, unique up to positive scaling, called the \textit{Perron vector} of $G$. For a real symmetry matrix $A$ of order $n,$ we may arrange its eigenvalues as $\lambda_1(A)\geq \lambda_2(A)\geq\cdots\geq\lambda_n(A)$.

\section{\normalsize Preliminaries}\label{s2}

\begin{lemma}[\cite{So1994}]\label{lem2.1}
Let $A$ and $B$ be real symmetric matrices of order $n,$ and let $1\leq i\leq n$ and $1\leq j\leq n.$ Then
\begin{align}\notag
\lambda_i(A)+\lambda_j(B)\leq\lambda_{i+j-n}(A+B),\,\text{if $i+j\geq n+1$,} \\ \notag
\lambda_i(A)+\lambda_j(B)\geq\lambda_{i+j-1}(A+B),\,\text{if $i+j\leq n+1$.}
\end{align}
In either of these inequality equality holds if and only if there exists a nonzero vector of order $n$ that is an eigenvector to each of the three eigenvalues involved.
\end{lemma}

Let \( M \) be an \( n \times n \) real symmetric matrix, and let \( \pi: V = V_1 \cup \cdots \cup V_\ell \) be a partition of \( V = \{1, 2, \ldots, n\} \). Corresponding to the partition \( \pi \), \( M \) can be partitioned into the following block matrix:
$$
M = \begin{pmatrix} M_{11} & \cdots & M_{1\ell} \\ \vdots & \ddots & \vdots \\ M_{\ell 1} & \cdots & M_{\ell\ell} \end{pmatrix}.
$$
The \textit{quotient matrix} of \( M \) with respect to \( \pi \) is the \( \ell \times \ell \) matrix \( B = (b_{ij}) \), where \( b_{ij} \) denotes the average row sum of the block \( M_{ij} \). The partition \( \pi \) is called \textit{equitable} if, for all \( i, j \in [\ell] \), each block \( M_{ij} \) has constant row sums.
\begin{lemma}[\cite{G1993}]\label{lem2.05}
Let $M$ be a real symmetric matrix and let $B$ be a quotient matrix of $M$ with respect to an equitable partition. Then the eigenvalues of $B$ are also the eigenvalues of $M.$ Furthermore, if $M$ is nonnegative and irreducible, then $\lambda_1(M)=\lambda_1(B).$
\end{lemma}

\begin{lemma}[\cite{YW2015}]\label{lem2.06}
Let $A, \hat{A} \in \mathbb{R}^{n \times n}$ be symmetric, with eigenvalues $\lambda_1 \geqslant \cdots \geqslant \lambda_n$ and $\hat{\lambda}_1 \geqslant \cdots \geqslant \hat{\lambda}_n$, respectively. Fix $j \in \{1, \ldots, n\}$, and assume that $\min\{\lambda_{j-1} - \lambda_j, \lambda_j - \lambda_{j+1}\} > 0$, where we define $\lambda_0 = \infty$ and $\lambda_{n+1} = -\infty$. If $v, \hat{v} \in \mathbb{R}^n$ satisfy $A v = \lambda_j v$ and $\hat{A} \hat{v} = \hat{\lambda}_j \hat{v}$ and $\hat{v}^\top v \geqslant 0$, then
$$
\| \hat{v} - v \| \leq \frac{2^{3/2} \rho(\hat{A} - A )}{\min(\lambda_{j-1} - \lambda_j, \lambda_j - \lambda_{j+1})}.$$
\end{lemma}
%\section{Candidate graphs and their spectra}

%We begin with a common notation.
Let $R(a,b;p,q)$ be the graph obtained
from $K_{a,b}$ by placing a matching of size $p$ in the $a$-part and a
matching of size $q$ in the $b$-part, where
$0\le p\le\lfloor a/2\rfloor$ and
$0\le q\le\lfloor b/2\rfloor$.  Put
\begin{equation}
 \alpha:=\frac a2-p,
 \qquad
 \zeta:=\frac b2-q.
\label{eq:alpha-eta}
\end{equation}
The quantities $\alpha$ and $\zeta$ are nonnegative half-integers.

\begin{lemma}
\label{lem:R-equation}
For $ab\ge 2,$ every graph $R(a,b;p,q)$ is $W_5$-free. Its spectral radius is the positive root of
\begin{equation}
 x^2(x-1)^2=(ax-2\alpha)(bx-2\zeta).
\label{eq:R-equation}
\end{equation}
Moreover, the graph $R(a,b;p,q)$ has only one eigenvalue greater than $1$.
\end{lemma}

\begin{proof}
Let $V(R(a,b;p,q))=A\cup B$ be a bipartition with $|A|=a,\,|B|=b$ and all vertices in $A$ are adjacent to all vertices in $B$. Then the neighborhood of a vertex in one part consists of the opposite part, which induces a matching, together with at most one matching partner of the original vertex. Such a graph contains no $C_4$, so
$R(a,b;p,q)$ is $W_5$-free.

Denote by $\rho(R(a,b;p,q))=\rho.$ Let ${\bf x}$ be the nonnegative unit eigenvector for $\rho$, and let $S_A$, $S_B$ be the sums of the coordinates of ${\bf x}$ corresponding to the vertices in $A$ and $B$, respectively. Then for a vertex $v\in A$ with $d_A(v)=1,$ one has $\rho x_v=x_v+S_B$, and so $x_v=S_B/(\rho-1)$; for a vertex $u\in A$ with $d_A(u)=0,$ one has $\rho x_u=S_B$, and so $x_u=S_B/\rho$. Hence
\begin{align}\label{eq2.3}
 S_A=2p\cdot \frac{S_B}{(\rho-1)}+2\alpha\cdot \frac{S_B}{\rho}=\frac{a\rho-2\alpha}{\rho(\rho-1)}S_B.
\end{align}
The symmetric identity for $S_B$ gives
$
 S_B=\frac{b\rho-2\zeta}{\rho(\rho-1)}S_A.
$
Together this with \eqref{eq2.3} gives
$
\rho^2(\rho-1)^2=(a\rho-2\alpha)(b\rho-2\zeta).
$
Hence, $\rho$ is the positive root of \eqref{eq:R-equation}.

Finally,
$R(a,b;p,q)$ is obtained from $K_{a,b}$ by adding two matchings, whose
union has the spectral radius at most $1$. Hence $\rho \ge \rho(K_{a,b})=\sqrt{ab}>1$, and Lemma~\ref{lem2.1} gives
$$\lambda_2(R(a,b;p,q))\le \lambda_2(K_{a,b})+1=1.$$
Therefore, $\rho$ is the only eigenvalue of $R(a,b;p,q)$ greater than $1$.
\end{proof}

\begin{proposition}
\label{prop:candidate-radii}
The spectral radii of the graphs presented in Theorem~\ref{thm:main} are as follows.
\begin{enumerate}[label=\textup{(\roman*)}]
 \item $\rho(E_{a,b})=1+\sqrt{ab}$.
 \item $\rho(O_{e,o})$ is the largest root of
 \begin{equation}
  f_{e,o}(x):=x^3-2x^2+(1-eo)x+e.
  \label{eq:O-cubic}
 \end{equation}
 \item $\rho(D_{a,b})$ is the largest root of
 \begin{equation}
  g_{a,b}(x):=x^3-2x^2+(1-ab)x+2a.
  \label{eq:D-cubic}
 \end{equation}
 \item $\rho(H_n)$ is the largest root of
 \begin{equation}
 \begin{split}
  h_n(x):={}&x^5-2x^4-(n^2+1)x^3+4x^2+(2n^2-n-2)x-n.
 \end{split}
 \label{eq:P-quintic}
 \end{equation}
\end{enumerate}
\end{proposition}
\begin{proof}
The first three assertions follow from Lemma~\ref{lem:R-equation} with
$(\alpha,\zeta)=(0,0)$, $(0,1/2)$, and $(0,1)$, respectively. %after canceling a positive factor $x$ where appropriate.

For $H_n$, partition its vertices into the pendant vertex, its neighbor
$v$, the matching partner of $v$, the remaining $n-2$ vertices in the
same part, and the opposite part. The quotient matrix of $A(H_n)$ corresponding to this equitable partition is
\[
 Q_n=\begin{pmatrix}
 0&1&0&0&0\\
 1&0&1&0&n\\
 0&1&0&0&n\\
 0&0&0&1&n\\
 0&1&1&n-2&1
 \end{pmatrix}.
\]
A direct determinant calculation gives
$\det(xI-Q_n)=h_n(x)$. By Lemma~\ref{lem2.05}, the largest root of $h_n$ is $\rho(H_n)$.
\end{proof}

%The following exact test is useful when evaluating
%\eqref{eq:main-maximum}.

%\begin{corollary}
%\label{cor:book-tests}
%Let $m$ be sufficiently large and odd, and put $\theta=\beta_m$.
%Then
%\begin{align*}
% \rho(O_{e,o})>\theta
% &\iff (m-1-eo)\theta+e-m+1<0,\\
% \rho(D_{a,b})>\theta
% &\iff (m-1-ab)\theta+2a-m+1<0.
%\end{align*}
%Equality is characterized by replacing both strict inequalities with
%equalities.  Also,
%\[
% \rho(E_{a,b})>\theta\iff 1+\sqrt{ab}>\theta.
%\]
%\end{corollary}

%\begin{proof}
%Use $\theta^2-\theta=m-1$ to reduce
%$f_{e,o}(\theta)$ and $g_{a,b}(\theta)$.  This gives the two displayed
%linear expressions.  By Lemma~\ref{lem:R-equation}, each polynomial has
%only one root greater than $1$, and it is increasing through that root.
%The last assertion follows from Proposition~\ref{prop:candidate-radii}.
%\end{proof}

%\section{Dense cores at the book threshold}

For a graph $G$ with at least one edge, define
\[
 \Phi(G):=\frac{\rho(G)}{\sqrt{e(G)}}.
\]
Fix $0<\varepsilon<1/100$.  Following
Fang, Lin and Zhai~\cite{FangLinZhai}, an edge-induced proper subgraph
$H\subsetneq G$ is $\varepsilon$-dense in $G$ if
\begin{align*}
 &(1-\varepsilon)e(G)<e(H)<e(G),\ \ \ \ \ \ \ \Phi(H)-\Phi(G)\ge
   \frac{\varepsilon(e(G)-e(H))}{2e(G)}.
\end{align*}
An $\varepsilon$-core is a graph with no proper $\varepsilon$-dense
subgraph and no isolated vertex.

\begin{lemma}[\cite{FangZhaiZhang}]\label{lem2.3} Let $G$ be an $\varepsilon$-core with sufficiently large size $m$, and let ${\bf x} = (x_v)$ be a nonnegative unit eigenvector for $\rho(G)$. If $\rho(G) \geq \sqrt{m}$, then
\begin{wst}
\item[{\rm (i)}] $x_u x_v > (1 - \varepsilon)/(4\sqrt{m})$ for every edge $uv \in E(G)$;
\item[{\rm (ii)}] $2m x_u^2 \geq (1 - 2\varepsilon)d_G(u)$ for every vertex $u \in V(G)$ with $d_G(u) < \varepsilon m$.
\end{wst}
\end{lemma}
Let $G$ and $H$ be two graphs (which may have distinct vertex sets). Define the \textit{distance} between $G$ and $H$ as
$$d(G,H) := |E(G) \setminus E(H)| + |E(H) \setminus E(G)|.$$
Given two disjoint vertex sets $U$ and $V$, we use $K_{U,V}$ to represent the complete
bipartite graph with parts $U$ and $V$. Li, Liu and Zhang showed the following edge-spectral stability result.
\begin{lemma}[\cite{LiLiuZhangStability}]\label{lem2.4}
Let $F$ be a graph with $\chi(F) = 3$. For every $\varepsilon > 0$, there exist $\delta > 0$ and $m_0$ such that if $G$ is an $F$-free graph with $m \geq m_0$ edges and $\rho^2(G) \geq (1 - \delta)m$, then there exist disjoint vertex sets $A, B \subseteq V(G)$ such that $d(G, K_{A,B}) \leq \varepsilon m$.
\end{lemma}

%Let $g^*(t)$ and $\beta(t)$ be the largest zeros of $f_1(x)=x^2-x-t$ and $f_2(x)=x^2-x-t+1$, respectively. Then
%\[
% g^*(t):=\frac{1+\sqrt{4t+1}}2,
% \qquad
% \beta(t):=\frac{1+\sqrt{4t-3}}2.
%\]

\section{Proof of Theorem~\ref{thm:main}}\label{s3}
In this section, fix odd $m$ with $m$ sufficiently large. Let $G$ be a $W_5$-free graph of size $m$ with no isolated vertex having the largest spectral radius. By the Perron-Frobenius theorem, the graph $G$ is connected. For convenience, denote by $\rho:=\rho(G),$ and let ${\bf x}=(x_v)$ be the unit Perron vector of $A(G)$.

Note that the graph $B_m$ is a $W_5$-free graph of size $m$, by the choice of $G$, one has
\begin{align}\label{eq3.1}
\rho\geq \rho(B_m)=\beta_m.
\end{align}
In the following, choose the constants in the hierarchy
\begin{equation}\label{eq:hierarchy}
 0<\varepsilon\ll\eta^2\ll\eta\ll1.
\end{equation}
\begin{lemma}%[Book-threshold core extraction]
\label{lem:book-core}
$G$ contains a connected $\varepsilon$-core $H$ of size $h$ such that
\begin{equation}\notag
 m-h=O_{\varepsilon}(1),
 \qquad
 \rho(H)\ge\beta_h.
\end{equation}
If $\rho(H)=\beta_h$, then $H=G$.
\end{lemma}

\begin{proof}
Since $G$ is connected, if $G$ itself is an $\varepsilon$-core, there is nothing to prove. In the following, we consider $G$ is not an $\varepsilon$-core. Starting with $G_0=G$, whenever $G_i$ is not an $\varepsilon$-core, choose an $\varepsilon$-dense subgraph $G_{i+1}$. Write
\(
 m_i=e(G_i),\,
 \Delta_i=m_i-m_{i+1}.
\)
Then $0<\Delta_i<\varepsilon m_i$ and
\begin{equation}\label{eq:dense-step}
 \Phi(G_{i+1})-\Phi(G_i)
 \geq \frac{\varepsilon\Delta_i}{2m_i}.
\end{equation}
Since the edge numbers decrease strictly in the sequence $G_0,G_1,G_2,\ldots$, the process terminates at an $\varepsilon$-core $H=G_r$. Put $h=m_r$ and $t=m-h$.

We first verify that $h=(1-o(1))m$. Let $q_i=\Delta_i/m_i$. Then one has $m_{i+1}=m_i(1-q_i)$ and
$0<q_i<\varepsilon$, Thus
$$
\log \frac{m_0}{m_r}=\log \prod_{i=0}^{r-1}\frac{m_i}{m_{i+1}}=\sum_{i=0}^{r-1}\log\frac{m_i}{m_{i+1}}
=-\sum_{i=0}^{r-1}\log(1-q_i).
$$
For $q\in [0,\varepsilon],$ we have $\log(1-q)=-\int_0^q\frac{dt}{1-t}\geq-\frac{q}{1-\varepsilon}$. Since $0<q_i<\varepsilon$, we obtain
\[
 \sum_{i=0}^{r-1}\frac{\Delta_i}{m_i}=\sum_{i=0}^{r-1}q_i
 \geq\sum_{i=0}^{r-1}(-(1-\varepsilon)\log(1-q_i))
 =(1-\varepsilon)\log\frac{m}{h}.
\]
On the other hand, since $\rho(H)^2\leq\operatorname{tr}A(H)^2=2h$, one has $\Phi(H)\leq\sqrt2$. Whereas
$\Phi(G)\geq\beta_m/\sqrt m=1+o(1)$. Summing
\eqref{eq:dense-step} therefore shows that
$$
\sqrt2-1+o(1)\geq\Phi(H)-\Phi(G)\geq\sum_{i=0}^{r-1}\frac{\varepsilon\Delta_i}{2m_i}
\geq \frac{\varepsilon(1-\varepsilon)}{2}\log\frac{m}{h},
$$
and so $\log(m/h)=O_{\varepsilon}(1)$, which implies $h\geq c_{\varepsilon}m$ for some
$c_{\varepsilon}>0$. If $h\leq(1-\alpha)m$ for some
$\alpha>0$, then summing
\eqref{eq:dense-step} shows that
\begin{align}\label{eq3.3}
 \Phi(H)-\Phi(G)\geq\sum_{i=0}^{r-1}\frac{\varepsilon\Delta_i}{2m_i}
 \geq \frac{\varepsilon}{2m}\sum_{i=0}^{r-1}\Delta_i
 =\frac{\varepsilon(m-h)}{2m}
 \geq\frac{\varepsilon\alpha}{2}.
\end{align}
But $h\geq c_{\varepsilon}m$, and $H$ is a $W_5$-free graph of size $h,$ so Theorem~\ref{thm:FZZ} gives
$\Phi(H)\leq g^*(h)/\sqrt h=1+o(1)$, while
$\Phi(G)\geq\beta_m/\sqrt m=1+o(1)$, a contradiction to \eqref{eq3.3}. Hence,
\begin{equation}\label{eq:h-close-m}
 h=(1-o(1))m.
\end{equation}

We next show $\rho(H)>\beta_h$.  Define
\[
 \psi_\beta(s):=\frac{\beta_s}{\sqrt s}
 =\frac1{2\sqrt s}+\sqrt{1-\frac3{4s}}.
\]
For all sufficiently large $s$,
\begin{equation}\notag
 \psi_\beta'(s)
 =-\frac1{4s^{3/2}}
  +\frac3{8s^2\sqrt{1-3/(4s)}}
 \quad\Longrightarrow\quad
 |\psi_\beta'(s)|\leq C s^{-3/2}
\end{equation}
for an absolute constant $C$.  By \eqref{eq:h-close-m}, we may assume
$h\geq m/2$.  Thus the mean value theorem gives
\begin{equation}\label{eq:beta-change}
 |\psi_\beta(h)-\psi_\beta(m)|
 \leq C_1\frac{t}{m^{3/2}}.
\end{equation}
Summing \eqref{eq:dense-step} and using $m_i\leq m$ gives
\begin{equation}\label{eq:accumulated-gain}
 \Phi(H)-\Phi(G)\geq\sum_{i=0}^{r-1}\frac{\varepsilon\Delta_i}{2m_i}
 \geq\frac{\varepsilon t}{2m}.
\end{equation}
Note that $t>0$, for sufficiently large $m$, the right-hand side of
\eqref{eq:accumulated-gain} is strictly larger than the right-hand side
of \eqref{eq:beta-change}. Consequently,
\[
 \Phi(H)\geq \Phi(G)+\frac{\varepsilon t}{2m}
 \geq\psi_\beta(m)+\frac{\varepsilon t}{2m}
 >\psi_\beta(h),
\]
so $\rho(H)>\beta_h$.

Since $\beta_h^2=h-1+\beta_h>h$ for large $h$, we have
$\rho(H)>\sqrt h$. Then Lemma \ref{lem2.3} shows $H$ is connected.

It remains to prove that $t$ is bounded. %This is the point at which the last display in the original proof must be replaced.
From Theorem~\ref{thm:FZZ}, \eqref{eq3.1} and \eqref{eq:accumulated-gain}, one has
\begin{equation}\label{eq:loss-start}
 \frac{\varepsilon t}{2m}
 \leq \frac{g^*(h)}{\sqrt h}-\frac{\beta_m}{\sqrt m}.
\end{equation}
Set $\psi_*(s)=g^*(s)/\sqrt s$. Then
\begin{align*}
 \frac{g^*(h)}{\sqrt h}-\frac{\beta_m}{\sqrt m}
 &=\bigl(\psi_*(h)-\psi_*(m)\bigr)
   +\frac{g^*(m)-\beta_m}{\sqrt m}.            % \label{eq:correct-split}
\end{align*}
Since
\[
 \psi_*'(s)
 =-\frac1{4s^{3/2}}
  -\frac1{8s^2\sqrt{1+1/(4s)}},
\]
and $h\geq m/2$, one has
\begin{equation}\notag
 0\leq\psi_*(h)-\psi_*(m)
 \leq C_2\frac{t}{m^{3/2}}.
\end{equation}
The remaining gap is explicit:
\begin{equation}\label{eq:threshold-gap}
 \frac{g^*(m)-\beta_m}{\sqrt m}
 =\frac{2}{\sqrt m\bigl(\sqrt{4m+1}+\sqrt{4m-3}\bigr)}
 \leq\frac{C_3}{m}.
\end{equation}
Combining \eqref{eq:loss-start}--\eqref{eq:threshold-gap} yields
\[
 \frac{\varepsilon t}{2m}
 \leq C_2\frac{t}{m^{3/2}}+\frac{C_3}{m}.
\]
After multiplication by $m$, the term $C_2t/\sqrt m$ can be absorbed
into the left-hand side for sufficiently large $m$.  Thus
$t\leq 4C_3/\varepsilon$, proving $m-h=O_{\varepsilon}(1)$.
\end{proof}

%We next record the structural part of the argument in \cite[Sections~4--5]{FangZhaiZhang} at the slightly lower book threshold.  The equality statement in the split case is included because it is needed below.

\begin{lemma}%[Book-threshold core dichotomy]
\label{lem:core-dichotomy}
%For every sufficiently small fixed $\eta>0$, there is a sufficiently
%small $\varepsilon>0$ such that the following holds.
Let $H$ be a connected $W_5$-free $\varepsilon$-core of size $h$ with $h$ sufficiently large and $\rho(H)\ge\beta_h$. Then exactly one of the following occurs.

\begin{wst}%[label=\textup{(\roman*)}]
 \item[{\rm (i)}] $h$ is odd, $H\cong B_h$, and $\rho(H)=\beta_h$.
 \item[{\rm (ii)}] There is a partition $V(H)=A\cup B$, with
 $a=|A|\le b=|B|<100a$, such that
 \begin{align}
  &a,b=\Theta(\sqrt h),\qquad ab=(1+o(1))h,
  \label{eq:balanced-sizes}\\
  &d_B(u)>(1-\eta)b\quad(u\in A),\qquad
    d_A(v)>(1-\eta)a\quad(v\in B),
  \label{eq:high-cross-degree}\\
  &\Delta(H[A])\le1,\qquad \Delta(H[B])\le1.
  \label{eq:core-matchings}
 \end{align}
 In addition, every coordinate of the unit Perron vector of $A(H)$ is $O(h^{-1/4})$.
\end{wst}
\end{lemma}

\begin{proof}
Write $\rho'=\rho(H)$ and let ${\bf z}=(z_v)$ be the unit Perron vector of $H$.
Since
\[
 \rho'^2\geq\beta_h^2=h-1+\beta_h>h,
\]
and Theorem~\ref{thm:FZZ} gives $\rho'\leq g^*(h)$, we have
$\rho'=(1+o(1))\sqrt h$. Note that $\chi(W_5)=3,$ by Lemma~\ref{lem2.4}, there are disjoint vertex subsets $A,B\subseteq V(H)$ for which
\[
 D:=d\bigl(H,K_{A,B}\bigr)=o(h).
\]
Here vertices outside $A\cup B$ are regarded as isolated in $K_{A,B}$. Among all such pairs choose $A,B$ first to minimize $D$ and,
subject to that, to minimize $|R|,$ where
\[
 R:=V(H)\setminus(A\cup B).
\]
Relabel the parts so that $a=|A|\leq b=|B|$. Every discrepancy between
$h$ and $ab$ is counted by $D$, i.e.,
\begin{equation}\label{eq:ab-close}
 |h-ab|\leq D=o(h)
 \quad\Longrightarrow\quad
 ab=(1+o(1))h.
\end{equation}
The minimality of $D$ and $|R|$ also gives
\begin{align}
 d_B(u)&\geq b/2 &&\text{  (for all } u\in A),                             \label{eq:min-A}\\
 d_A(v)&\geq a/2 &&\text{  (for all } v\in B),                             \label{eq:min-B}\\
 d_A(w)&<a/2,\quad d_B(w)<b/2 &&\text{  (for all } w\in R).                \label{eq:min-R}
\end{align}
In fact, if there is a vertex $u\in A$ such that $d_B(u)<b/2,$ choose $A'=A\setminus \{u\}$. Then
\begin{align*}
d(H,K_{A',B})&=|E(H)\setminus E(K_{A',B})|+|E(K_{A',B})\setminus E(H)| \\
&=(|E(H)\setminus E(K_{A,B})|+d_B(u))+(|E(K_{A,B})\setminus E(H)|-(b-d_B(u)))\\
&=D-(b-2d_B(u))<D,
\end{align*}
a contradiction to the minimality of $D$. Therefore, \eqref{eq:min-A} holds. By symmetry, one sees \eqref{eq:min-B} holds. If there is a vertex $w\in R$ such that
$d_B(w)\geq b/2$, choose $A''=A\cup \{w\}.$ Then
\begin{align*}
d(H,K_{A'',B})&=|E(H)\setminus E(K_{A'',B})|+|E(K_{A'',B})\setminus E(H)| \\
&=(|E(H)\setminus E(K_{A,B})|-d_B(w))+(|E(K_{A,B})\setminus E(H)|+(b-d_B(w)))\\
&=D-(2d_B(w)-b)\leq D.
\end{align*}
Now $d(H,K_{A'',B})<D$ contradicts the minimality of $D$, and $d(H,K_{A'',B})=D$ contradicts the minimality of $R$. Therefore, for all $w\in R$, one has
$d_B(w)< b/2,$ By symmetry, for all $w\in R$, one has $d_A(w)< a/2.$

Define the good sets
\[
 A_0=\{u\in A:d_B(u)\geq(1-\eta^2)b\},
 \qquad
 B_0=\{v\in B:d_A(v)\geq(1-\eta^2)a\}.
\]
Since $|E(K_{A,B})\setminus E(H)|\leq D=o(ab)$,
\begin{equation}\label{eq:good-sets}
 |A\setminus A_0|=o(a),
 \qquad
 |B\setminus B_0|=o(b).
\end{equation}

We shall repeatedly use the following local consequence:
\begin{equation}\label{eq:local-obstruction}
 d_B(w)>4\eta^2b
 \quad\Longrightarrow\quad
 d_A(w)\leq\eta^2a+1\quad (\forall w\in V(H)).
\end{equation}
Indeed, if there is a vertex $w\in V(H)$ such that $d_B(w)>4\eta^2b$ and $d_A(w)>\eta^2a+1$. Then
$$|N_A(w)\cap A_0|\geq d_A(w)-|A\setminus A_0|> (\eta^2+o(1))a+1>1,$$
and so $N_A(w)\cap A_0$ contains two distinct vertices
$u_1,u_2$. Note that $b\geq a$ and \eqref{eq:ab-close} imply $b$ is sufficiently large for sufficiently large $h$. Then
\[
 |N_B(w)\cap N_B(u_1)\cap N_B(u_2)|\geq |N_B(w)|+ |N_B(u_1)|+ |N_B(u_2)|-2b\geq 2\eta^2b\geq 2,
\]
and so there are at least two vertices $v_1,v_2$ in $N_B(w)\cap N_B(u_1)\cap N_B(u_2)$, give the cycle $u_1v_1u_2v_2u_1$ inside $H[N(w)]$, so $H[N(w)\cup \{w\}]$ contains a $W_5,$ a contradiction. Hence \eqref{eq:local-obstruction} holds. If $a$ is sufficiently large, then the symmetric assertion also holds. In particular, by \eqref{eq:min-A}, \eqref{eq:min-B}, and \eqref{eq:local-obstruction}, when $a$ is sufficiently large, one has
\begin{equation}\label{eq:small-internal-degree}
 d_A(u)\leq2\eta^2a\quad(u\in A),
 \qquad
 d_B(v)\leq2\eta^2b\quad(v\in B).
\end{equation}

\begin{claim}\label{claim1}
There is an absolute constant $C>0$ such that
\begin{align}
 z_u^2&\geq(1-C\eta^2)\frac{b}{2h} &&(u\in A_0),         \label{eq:lower-A0}\\
 z_v^2&\geq(1-C\eta^2)\frac{a}{2h} &&(v\in B_0).         \label{eq:lower-B0}
\end{align}
\end{claim}
\begin{proof}
To prove \eqref{eq:lower-B0}, take $v\in B_0$. We have
$d_H(v)\leq a+D<\varepsilon h$, and therefore Lemma~\ref{lem2.3} gives
\begin{align}\label{eq3.10}
 z_v^2
 \geq(1-2\varepsilon)\frac{d_H(v)}{2h}
 \geq(1-2\varepsilon)(1-\eta^2)\frac{a}{2h}
 \geq(1-3\eta^2)\frac{a}{2h},
\end{align}
as desired.

Now let $u\in A_0$. By \eqref{eq:good-sets},
$d_{B_0}(u)\geq(1-2\eta^2)b$. Then $A(H){\bf z}=\rho{\bf z}$ and \eqref{eq3.10} give
\[
 \rho' z_u
 \geq\sum_{v\in N_{B_0}(u)}z_v
 \geq(1-2\eta^2)b
       \sqrt{(1-3\eta^2)\frac{a}{2h}}.
\]
Squaring and using \eqref{eq:ab-close} and
$\rho'^2=(1+o(1))h$ proves \eqref{eq:lower-A0}.
\end{proof}

Summing \eqref{eq:lower-A0} and \eqref{eq:lower-B0}, and using \eqref{eq:ab-close} and
\eqref{eq:good-sets}, gives
\begin{equation}\label{eq:mass-halves}
 \min\left\{\sum_{u\in A_0}z_u^2,
             \sum_{v\in B_0}z_v^2\right\}
 \geq\frac12\bigl(1-C\eta^2-o(1)\bigr).
\end{equation}
Consequently, for every $S\subseteq A$ and $T\subseteq B$,
\begin{align}
 \sum_{u\in S}z_u^2&\leq1-\sum_{u\in B_0}z_u^2-\sum_{u\in A_0\setminus S}z_u^2
 \leq\frac{|S|}{2a}+C_1\eta^2+o(1),                    \label{eq:subset-A}\\
 \sum_{v\in T}z_v^2&\leq1-\sum_{v\in A_0}z_v^2-\sum_{v\in B_0\setminus T}z_v^2
 \leq\frac{|T|}{2b}+C_1\eta^2+o(1).                    \label{eq:subset-B}
\end{align}
%For completeness, the first estimate follows by putting
%$\lambda_A=(1-C\eta^2)b/(2h)$ and observing that
%\[
% \sum_{u\in S}z_u^2
% \leq\left(1-\sum_{v\in B_0}z_v^2\right)
%      -\lambda_A|A_0\setminus S|.
%\]
%Now use $|A_0|=(1-o(1))a$, $ab=(1+o(1))h$, and \eqref{eq:mass-halves}.  The proof of \eqref{eq:subset-B} is identical.

\begin{claim}\label{claim2}
$R=\emptyset.$
\end{claim}
\begin{proof}
Suppose to the contrary that $R\neq \emptyset$, then take a vertex $w\in R$. Note that every edge incident with $w$ is in $E(H)\setminus E(K_{A,B})$, so $d_H(w)\leq D=o(h)$.  Now by Lemma~\ref{lem2.3} and the Cauchy--Schwarz inequality, one has
\begin{equation}\label{eq:R-lower}
 \sum_{u\in N(w)}z_u^2
 \geq\frac1{d_H(w)}\left(\sum_{u\in N(w)}z_u\right)^2
 =\frac{\rho'^2z_w^2}{d_H(w)}
 \geq(1-2\varepsilon)\frac{\rho'^2}{2h}
 >\frac{1-2\varepsilon}{2}.
\end{equation}
On the other hand, by \eqref{eq:mass-halves},
\begin{equation}\label{eq:R-mass}
 \sum_{u\in R}z_u^2=1-\sum_{u\in A}z_u^2-\sum_{u\in B}z_u^2\leq C\eta^2+o(1).
\end{equation}
If $d_B(w)\leq4\eta^2b$, then
\eqref{eq:min-R}, \eqref{eq:subset-A}, \eqref{eq:subset-B}, and
\eqref{eq:R-mass} imply
\[
 \sum_{u\in N(w)}z_u^2
 \leq\frac{d_A(w)}{2a}+\frac{d_B(w)}{2b}
      +O(\eta^2)+o(1)
 \leq\frac14+O(\eta^2)+o(1),
\]
a contradiction to \eqref{eq:R-lower}. Hence $d_B(w)>4\eta^2b$, and so
\eqref{eq:local-obstruction} gives $d_A(w)\leq\eta^2a+1$.
If $a\geq4$, then
\[
 \frac{d_A(w)}{2a}\leq\frac18+\frac{\eta^2}{2},
 \qquad
 \frac{d_B(w)}{2b}<\frac14,
\]
again contradicting \eqref{eq:R-lower}.  If $a<4$, integrality and
$d_A(w)<a/2$ give
\[
 \frac{d_A(w)}{2a}\leq\frac14-\frac1{4a}
 \leq\frac14-\frac1{12},
\]
and the same contradiction follows. Thus $R=\emptyset.$
\end{proof}

By Claim~\ref{claim2}, $V(H)=A\cup B$.
Put
\[
 z_A^*=\max_{u\in A}z_u,
 \qquad
 z_B^*=\max_{v\in B}z_v,
 \qquad
 z^*=\max\{z_A^*,z_B^*\}.
\]
\begin{claim}\label{claim3}
If $a$ is sufficiently large, then the numbers $z_A^*,\,z_B^*$ and $z^*$ satisfy
\begin{align}
 z_A^*&\leq(1+O(\eta^2)+o(1))\sqrt{\frac{b}{2h}}
       +(O(\eta^2)+o(1))z^*,                              \label{eq:max-A}\\
 z_B^*&\leq(1+O(\eta^2)+o(1))\sqrt{\frac{a}{2h}}
       +(O(\eta^2)+o(1))z^*.                              \label{eq:max-B}
\end{align}
\end{claim}
\begin{proof}
For any $w\in V(H)$, applying the eigen-equation twice gives
\begin{align}\notag
 \rho'^2z_w&=\sum_{u\in N(w)}\sum_{v\in N(u)}z_v \\ \notag
 &=\sum_{u\in N_B(w)}\sum_{v\in N_B(u)}z_v+\sum_{u\in N_A(w)}\sum_{v\in N_A(u)}z_v
 +\sum_{u\in N_A(w)}\sum_{v\in N_B(u)}z_v+\sum_{u\in N_B(w)}\sum_{v\in N_A(u)}z_v \\
 &\leq d_A(w)b z_B^*+d_B(w)a z_A^*+2Dz^*. \label{eq:two-step}
\end{align}

Choose $u_0\in A_0$ with $z_{u_0}=\min\{z_u|u\in A_0\}$. From \eqref{eq:ab-close}, \eqref{eq:good-sets}, and
\eqref{eq:mass-halves}, one gets
\begin{align}\notag
 z_{u_0}^2\leq\frac{1}{a-o(a)}\sum_{u\in A_0}z_u^2
 \leq \frac{1}{a-o(a)}\left(1-\sum_{u\in B_0}z_u^2\right)\leq \frac{1}{2(a-o(a))}(1+C\eta^2+o(1)),
\end{align}
and so
\begin{align}\label{eq:small-good-coordinate}
z_{u_0}\leq\sqrt{\frac{1}{2(a-o(a))}(1+C\eta^2+o(1))}
=(1+O(\eta^2)+o(1))\sqrt{\frac{b}{2h}}.
\end{align}

Let $u_1\in A$ satisfy $z_{u_1}=z_A^*$. Subtracting the
eigen-equations $A(H){\bf z}=\rho'{\bf z}$ at $u_1$ and $u_0$, then applying the eigen-equation $A(H){\bf z}=\rho'{\bf z}$
once more, yields
\begin{align}\notag
 \rho'^2(z_{u_1}-z_{u_0})&=\sum_{u\in N(u_1)}\rho' z_u-\sum_{u\in N(u_0)}\rho' z_u \\ \notag
 &\leq\sum_{u\in N(u_1)\setminus N(u_0)}\rho' z_u \\ \notag
 &\leq\sum_{u\in N(u_1)\setminus N(u_0)}d_H(u)z^*\\ \notag
 &=\sum_{u\in N_A(u_1)\setminus N_A(u_0)}d_A(u)z^*
 +\sum_{u\in N_B(u_1)\setminus N_B(u_0)}d_B(u)z^*\\ \notag
 &\quad +\sum_{u\in N_A(u_1)\setminus N_A(u_0)}d_B(u)z^*
 +\sum_{u\in N_B(u_1)\setminus N_B(u_0)}d_A(u)z^*\\ \label{eq3.8}
&\leq (2D+d_A(u_1)b+|N_B(u_1)\setminus N_B(u_0)|a)z^*.
\end{align}
By \eqref{eq:small-internal-degree},
\[
 d_A(u_1)b+|N_B(u_1)\setminus N_B(u_0)|a\leq d_A(u_1)b+(b-d_B(u_0))a
 \leq2\eta^2ab+\eta^2ab.
\]
Then \eqref{eq3.8} gives
\begin{align}\label{eq3.9}
 \rho'^2(z_{u_1}-z_{u_0})\leq (2D+3\eta^2ab)z^*=(3\eta^2+o(1))abz^*.
\end{align}
Since $\rho'^2>h$, $ab=(1+o(1))h$, together with \eqref{eq:small-good-coordinate} and \eqref{eq3.9}, we obtain
$$
z_A^*=z_{u_1}\leq (1+O(\eta^2)+o(1))\sqrt{\frac{b}{2h}}
       +(O(\eta^2)+o(1))z^*.
$$

This proves \eqref{eq:max-A}. Similarly, \eqref{eq:max-B} can be proved by symmetry.
\end{proof}

We complete the proof of Lemma~\ref{lem:core-dichotomy} by considering the following two cases.

\textbf{Case 1: $b<100a$.} In this case, \eqref{eq:ab-close} implies
$a,b=\Theta(\sqrt h)$, so in particular $a$ is sufficiently large. If
$z^*=z_A^*$, then by \eqref{eq:max-A}, one has
\begin{align}
 z^*\leq(1+O(\eta^2)+o(1))\sqrt{\frac{b}{2h}}=O(h^{-1/4}).          \label{eq:balanced-max-A}
\end{align}
If $z^*=z_B^*$, then by \eqref{eq:max-B}, one has
\begin{align}
 z^*\leq(1+O(\eta^2)+o(1))\sqrt{\frac{a}{2h}}=O(h^{-1/4}).          \label{eq:balanced-max-B}
\end{align}
By \eqref{eq:balanced-max-A} and \eqref{eq:balanced-max-B}, for all $v\in V(H),$
\begin{align}\notag
 z_v=O(h^{-1/4}).
\end{align}

Now we consider the number of neighbors of a vertex $u\in A$ in $B$, and the number of neighbors of a vertex $v\in B$ in $A$. Fix $u\in A$ and set
\[
 \alpha_u=\frac{d_A(u)}a,
 \qquad
 \gamma_u=\frac{d_B(u)}b.
\]
By \eqref{eq:min-A} and \eqref{eq:small-internal-degree},
$\gamma_u\geq1/2$ and $\alpha_u\leq2\eta^2$.  Since
$d_H(u)=O(\sqrt h)<\varepsilon h$, by Lemma~\ref{lem2.3}, one has
\begin{align}\label{eq:cross-lower}
 z_u\geq\sqrt{(1-2\varepsilon)\frac{d_B(u)}{2h}}
 =(1-2\varepsilon)\sqrt{\gamma_u}\sqrt{\frac{b}{2h}}.
\end{align}
On the other hand, \eqref{eq:two-step},
\eqref{eq:balanced-max-A}, and \eqref{eq:balanced-max-B} give
\begin{equation}\label{eq:cross-upper}
 z_u
 \leq\bigl(\gamma_u+O(\eta^2)+o(1)\bigr)
          \sqrt{\frac{b}{2h}}.
\end{equation}
If $\gamma_u\leq1-\eta$, then, because the function
$f(t)=\sqrt t-t$ is decreasing on $[1/2,1]$,
\[
 \sqrt{\gamma_u}-\gamma_u
 \geq\sqrt{1-\eta}-(1-\eta)>\eta/3
\]
for small $\eta$.  This contradicts
\eqref{eq:cross-lower}--\eqref{eq:cross-upper} under the hierarchy
\eqref{eq:hierarchy}.  Hence
$d_B(u)>(1-\eta)b$ for every $u\in A$.  The symmetric argument proves
$d_A(v)>(1-\eta)a$ for every $v\in B$.

Finally, we are going to show $\Delta(H[A])\le1$ and $\Delta(H[B])\le1.$ Suppose to the contrary that there is a vertex $u\in A$ such that $d_A(u)\geq 2,$ then choose two distinct vertices $u_1,u_2\in N_A(u)$, one has
\[
 |N_B(u)\cap N_B(u_1)\cap N_B(u_2)|\geq d_B(u)+d_B(u_1)+d_B(u_2)-2b>(1-3\eta)b\geq2.
\]
Hence there are two distinct vertices $v_1,v_2$ in $N_B(u)\cap N_B(u_1)\cap N_B(u_2)$ form the cycle $u_1v_1u_2v_2u_1$ inside $H[N(u)]$, giving a $W_5$ in $H$, a contradiction.  Thus $\Delta(H[A])\leq1$, and symmetrically $\Delta(H[B])\leq1$.
This proves all assertions in alternative (ii).

\textbf{Case 2: $b\geq100a$.} In this case, choose $u^*\in V(H)$ with $z_{u^*}=z^*$. We first establish
\begin{align}
 u^*&\in A,                                               \label{eq:split-center}\\
 d_B(u^*)&\geq(1-3\eta^2)b,                              \label{eq:split-degree}\\
 z_v&\leq\tfrac14z^* &&(\text{ for all $v\in B$}),                         \label{eq:split-Bcoord}\\
 d_B(u)&>\tfrac{11}{20}b
   &&(\text{ for all $u\in A$ with $z_u>\tfrac35z^*$}).               \label{eq:split-large-A}
\end{align}

First we assume $a$ is sufficiently large. If $z^*=z_B^*$, then \eqref{eq:max-B} gives
$z^*\leq(1+O(\eta^2)+o(1))\sqrt{a/(2h)}$. This contradicts
\eqref{eq:lower-A0}, because $b/a\geq100$.  Hence
$u^*\in A$.  Equations \eqref{eq:lower-A0} and \eqref{eq:max-A} also
give
\begin{equation}\label{eq:split-center-size}
 z^*=(1+O(\eta^2)+o(1))\sqrt{\frac{b}{2h}}.
\end{equation}
If \eqref{eq:split-degree} failed, then
\eqref{eq:small-internal-degree} and \eqref{eq:two-step} would give
\(
 \rho'^2z^*
 \leq\bigl((1-\eta^2)ab+o(h)\bigr)z^*<hz^*,
\)
contrary to $\rho'^2>h$.

Next,
\eqref{eq:max-B}, \eqref{eq:split-center-size}, and $b/a\geq100$
give \eqref{eq:split-Bcoord}.

Finally, let $u\in A$ with $z_u>3z^*/5$, write $\gamma_u=d_B(u)/b$.  From
\eqref{eq:small-internal-degree}, \eqref{eq:two-step}, and
\eqref{eq:split-Bcoord},
\[
 z_u\leq\bigl(\gamma_u+\tfrac12\eta^2+o(1)\bigr)z^*.
\]
Thus $\gamma_u>11/20$ for small $\eta$ and large $h$, proving
\eqref{eq:split-large-A}.

It remains to justify the same statements if $a$ stays bounded. Then
$b=\Theta(h)$. Note that $\rho'=(1+o(1))\sqrt{h}=(1+o(1))\sqrt{ab}$. Put
\[
 p=\sum_{u\in A}z_u^2,
 \quad
 q=\sum_{v\in B}z_v^2,
 \quad
 M=A(H)-A(K_{A,B}).
\]
Then $p+q=1$, $\rho(M)\leq \sqrt{2D}.$ Using Cauchy--Schwarz inequality in the Rayleigh quotient gives
\begin{align*}
 \rho'=z^TA(H)z&=z^TA(K_{A,B})z+z^TMz\leq 2\left(\sum_{u\in A}z_u\right)\left(\sum_{v\in B}z_v\right)+\rho(M)\\
 &\leq2\sqrt{ap}\sqrt{bq}+\sqrt{2D}
 =(2\sqrt{pq}+o(1))\sqrt{ab}.
\end{align*}
It follows $\rho'=(1+o(1))\sqrt{ab}$ that $p=q=1/2+o(1)$, so $z_A^*$ is bounded below by a
positive constant.  For $v\in B$,
\[
 \rho' z_v
 \leq\sum_{u\in A}z_u+
      \sum_{w\in N(v)\setminus A}z_w
 \leq\sqrt{ap}+\sqrt D,
\]
and therefore $z_v=o(1)$.  This proves
\eqref{eq:split-center}-\eqref{eq:split-Bcoord}.

Now put
\[
 U=N_H(u^*),
 \qquad
 W=V(H)\setminus\bigl(\{u^*\}\cup U\bigr).
\]
\begin{claim}\label{claim4}
There is at most one vertex $u_0\in N_A(u^*)$ with the
property: every edge $uv\in E(H[U])$ not incident with
$u_0$ satisfies
\begin{equation}\label{eq:nonpositive-surplus}
 z_u+z_v\leq z^*.
\end{equation}
\end{claim}
\begin{proof}
If $d_A(u^*)\leq 1$, then the claim holds clearly. In the following, we consider $d_A(u^*)\geq 2.$ For $u\in N_A(u^*)$ set
\[
 Y_u=N_B(u^*)\cap N_B(u).
\]
By \eqref{eq:min-A} and \eqref{eq:split-degree},
$|Y_u|\geq(1/2-3\eta^2)b$. On the other hand, for any two distinct vertices $u,u'$ in
$N_A(u^*)$, we have $|Y_u\cap Y_{u'}|\leq1$; otherwise a $4$-cycle appears
inside $H[N(u^*)]$, and so $H[N(u^*)\cup\{u^*\}]$ contains a $W_5$, a contradiction.  Inclusion--exclusion now shows that $d_A(u^*)\leq2$, and so $d_A(u^*)=2$. Assume
$N_A(u^*)=\{u_0,u_1\}$, label them so that
$d_B(u_0)\geq d_B(u_1)$.  The bound
$|Y_{u_0}\cap Y_{u_1}|\leq1$ gives
\[
 d_B(u^*)+d_B(u_0)+d_B(u_1)\leq2b+1,
\]
and hence $d_B(u_1)<11b/20$.  By
\eqref{eq:split-large-A}, $z_{u_1}\leq3z^*/5$.  An edge not incident
with $u_0$ either lies in $B$ or joins $u_1$ to $B$, so the sum of its
endpoint coordinates is at most
$(3/5+1/4)z^*<z^*$.
\end{proof}

For $w\in W$,
define
\[
 f(w)=d_U(w)(z^*-z_w)+\frac12d_W(w)z^*\geq0.
\]
Applying the eigen-equation $A(H){\bf z}=\rho'{\bf z}$ twice at $u^*$ and grouping edges according
to $U$ and $W$ gives
\begin{equation}\notag
 \rho'^2z^*
 =hz^*+\sum_{uv\in E(H[U])}(z_u+z_v-z^*)
       -\sum_{w\in W}f(w).
\end{equation}
Subtracting $\rho' z^*$ yields
\begin{equation}\label{eq:residual}
 (\rho'^2-\rho'-h)z^*
 =\sum_{uv\in E(H[U])}(z_u+z_v-z^*)
  -\rho' z^*-\sum_{w\in W}f(w).
\end{equation}

By Claim~\ref{claim4}, there is at most one vertex $u_0\in N_A(u^*)$ with the
property: every edge $uv\in E(H[U])$ not incident with
$u_0$ satisfies \eqref{eq:nonpositive-surplus}. If there is no such vertex, then
\eqref{eq:nonpositive-surplus} and \eqref{eq:residual} imply
$\rho'^2-\rho'-h\leq-\rho'<-1$, contradicting
\(
 \rho'^2-\rho'-h
 \geq\beta_h^2-\beta_h-h=-1.
\)
Thus such vertex $u_0$ exists.

Now
\begin{align*}
 \sum_{v\in N_U(u_0)}(z_{u_0}+z_v-z^*)
 &\leq\sum_{v\in N_U(u_0)}z_v\leq\rho' z_{u_0}-z^*
 \leq(\rho'-1)z^*.
\end{align*}
It follows from \eqref{eq:residual} that
\begin{equation}\label{eq:split-defect}
 \rho'^2-\rho'-h\leq-1.
\end{equation}
The reverse inequality follows from $\rho'\geq\beta_h$. Hence equality
holds in \eqref{eq:split-defect} and in every inequality used to derive
it. In particular,
\begin{equation}\notag
 \rho'=\beta_h,
 \qquad z_{u_0}=z^*,
 \qquad N_W(u_0)=\emptyset,
 \qquad f(w)=0~~(\forall w\in W),
\end{equation}
and every edge $uv$ of $H[U]$ not incident with $u_0$ satisfies $z_u+z_v=z^*$.

The equality $f(w)=0$ shows that $W$ is independent. If $W\neq\emptyset$, then the connectedness of $H$ gives $d_U(w)>0$ for every $w\in W$, and
then $f(w)=0$ implies $z_w=z^*$. Comparing the eigen-equations $A(H){\bf z}=\rho'{\bf z}$ at $w$ and $u^*$ gives $N(w)=U$,
contradicting $N_W(u_0)=\emptyset$. Thus $W=\emptyset$.

Comparing the eigen-equations $A(H){\bf z}=\rho'{\bf z}$ at $u_0$ and $u^*$, using
$z_{u_0}=z^*$, shows that $u_0$ is adjacent to every vertex of
$U\setminus\{u_0\}$. Put $S=U\setminus\{u_0\}$. Then $\Delta(H[S])\leq1$, otherwise a vertex of $S$ has two neighbors in $S$, these three vertices together with
$u_0$ would give a $C_4$ in $H[N(u^*)]$, and so $H[N(u^*)\cup\{u^*\}]$ contains a $W_5$, a contradiction.

Now if $E(H[S])\neq\emptyset$, then $E(H[S])$ consists of a matching, and for $uv\in E(H[S])$, one has $z_u+z_v=z^*$. The eigen-equations $A(H){\bf z}=\rho'{\bf z}$ at
$u$ and $v$ are
\[
 \rho' z_u=2z^*+z_v,
 \qquad
 \rho' z_v=2z^*+z_u.
\]
Their difference gives $z_u=z_v=z^*/2$, and substitution then gives
$\rho'=5$, impossible for sufficiently large $h$. Therefore $E(H[S])=\emptyset$, and so
\(
 H=K_2\vee \overline{K_{|S|}}=B_h.
\)
In particular, $h=2|S|+1$ is odd and $\rho'=\beta_h$. This is
alternative (i).

The two alternatives are exclusive, % Indeed, in a large book all but two vertices have degree two, whereas both parts in alternative (ii) have order $\Theta(\sqrt h)$ and every vertex has more than a $(1-\eta)$-fraction of the opposite part as neighbors.  Thus a book cannot satisfy alternative (ii).
and so the proof of this lemma is complete.
\end{proof}

For a graph $G'$, define
\begin{equation}\notag
 \gamma(G'):=e(G')-\bigl(\rho(G')^2-\rho(G')\bigr).
\end{equation}
Theorem~\ref{thm:FZZ} says that for every $W_5$-free graph $G'$ with $e(G')$ sufficiently
large, one has $\gamma(G')\ge0$. Together with \eqref{eq:book-defect} and \eqref{eq3.1}, one has
\begin{equation}
 0\leq\gamma(G)\leq1.
\label{eq:defect-one}
\end{equation}

Let $H$ be a $\varepsilon$-core of $G$ supplied by Lemma~\ref{lem:book-core}, then $H$ is a connected $W_5$-free $\varepsilon$-core of size $h$ with $h$ sufficiently large and $\rho(H)\ge\beta_h$, and so exactly one of the alternatives~(i) and (ii) of Lemma~\ref{lem:core-dichotomy} holds for $H$. If the alternative~(i) of Lemma~\ref{lem:core-dichotomy} holds for $H$, then by Lemma~\ref{lem:book-core}, $G=H\cong B_m.$

In the rest lemmas of this section, we assume the alternative~(ii) of Lemma~\ref{lem:core-dichotomy} holds for $H$. Let
$V':=V(H)$ and let
\(
 H':=G[V'],\  t:=e(G)-e(H').
\)

%\begin{lemma}
%\label{lem:closure-matchings}
%If the alternative~(ii) of Lemma~\ref{lem:core-dichotomy} holds, then the partition $V'=A\cup B$ supplied by Lemma~\ref{lem:core-dichotomy} satisfies
%\[
% \Delta(H[A])\le1,\qquad \Delta(H[B])\le1.
%\]
%\end{lemma}

%\begin{proof}
%Suppose to the contrary that a vertex $u\in A$ has two neighbors $u_1,u_2$ in $H[A]$. By
%\eqref{eq:high-cross-degree}, the three vertices $u,u_1,u_2$ have at least $(1-3\eta)b\ge2$ common neighbors $v_1,v_2$ in $B$, using only edges of $X$.  Then
%$u_1v_1u_2v_2u_1$ is a $C_4$ in $G[N_G(u)]$, a contradiction.  The argument in $B$ is identical.
%\end{proof}

\begin{lemma}
\label{lem:one-edge-leakage}
The quantity $t\le1$. Further, if
$t=1$, then the unique edge outside $H'$ is a pendant edge attached to some vertex of $H$.
\end{lemma}

\begin{proof}
Lemma~\ref{lem:book-core} gives $t=O_\varepsilon(1)$. Let $x^*=\max_{v\in V(G)}x_v$. A vertex $w$ outside $V'$ has degree at most $t=O(1)$, so $x_w\neq x^*$, since otherwise the
eigen-equation $A(G){\bf x}=\rho {\bf x}$ would give
$$
\rho x_w=\sum_{u\sim w}x_u\le d_G(w)x_w\le tx_w,
$$
and so $\rho\le t=O_\varepsilon(1)$, a contradiction. Thus a vertex $u^\ast$ in $V'$ satisfies with $x_{u^\ast}=x^*$.

By Lemma~\ref{lem:core-dichotomy}(ii), its degree is
$O(\sqrt {e(H')})=O(\sqrt m)$. By eigen-equation $A(G){\bf x}=\rho {\bf x}$ and Cauchy--Schwarz inequality, one has
\begin{align*}
\rho x^*=\sum_{v\sim u^\ast}x_v\leq \sqrt{d_G(u^\ast)}\left(\sum_{v\sim u^\ast}x_v^2\right)^{1/2}\leq \sqrt{d_G(u^\ast)},
\end{align*}
and so
\begin{align}
 x^*\leq \frac{\sqrt{d_G(u^\ast)}}{\rho}=\frac{O(m^{1/4})}{\Omega(m^{1/2})}=O(m^{-1/4}).
\label{eq:max-coordinate-G}
\end{align}
Every vertex $w\notin V'$ then satisfies
\begin{equation}
 x_w\le\frac{d_G(w)x^*}{\rho}=O(m^{-3/4}).
\label{eq:outer-coordinate}
\end{equation}

Writing $x_{V'}$ for the restriction of ${\bf x}$ to $V'$, the Rayleigh quotient
and \eqref{eq:max-coordinate-G}--\eqref{eq:outer-coordinate} give
\begin{align*}
 \rho\le x_{V'}^{T}A(H')x_{V'}+O(m^{-1})\le\rho(H')\lVert x_{V'}\rVert^2+O(m^{-1})
 \le\rho(H')+O(m^{-1}).
\end{align*}
Consequently,
\begin{equation}
 \bigl(\rho^2-\rho\bigr)
 -\bigl(\rho(H')^2-\rho(H')\bigr)=O(m^{-1/2}).
\label{eq:defect-gain-small}
\end{equation}
Note that $H'$ is $W_5$-free, by Theorem~\ref{thm:FZZ}, $\gamma(H')\ge0$. Hence
\begin{align}\label{eq3.026}
 \gamma(G)=\gamma(H')+t-O(m^{-1/2})\ge t-o(1).
\end{align}
Together with \eqref{eq:defect-one}, this forces $t\le1$.

If $t=1$, connectedness implies that the unique edge not contained in
$H'$ joins one vertex of $V'$ to one vertex outside $V'$. It is therefore
a pendant edge.
\end{proof}

%\section{Defect quantization in the balanced core}

%We now treat the case $t=0$, so $G=H$.  The core $X$ and $G$ have the same vertex set and differ in only $O(1)$ edges.  This lets us transfer the lower Perron-coordinate estimates from $X$ to $G$.

\begin{lemma}%[Uniform Perron localization]
\label{lem:uniform-localization}
Let ${\bf y}$ be the Perron vector of $H'$, then
\begin{equation}
 y_u\ge(1-o(1))\sqrt{\frac{b}{2m}}\quad(u\in A),
 \qquad
 y_v\ge(1-o(1))\sqrt{\frac{a}{2m}}\quad(v\in B).
\label{eq:lower-coordinates-G}
\end{equation}
%Here the error can be made arbitrarily small by first choosing the constant hierarchy in Lemma~\ref{lem:core-dichotomy} sufficiently strong and then taking $m$ sufficiently large.
\end{lemma}

\begin{proof}
Since the alternative~(ii) of Lemma~\ref{lem:core-dichotomy} holds for $H$, \eqref{eq:balanced-sizes} and \eqref{eq:core-matchings} imply the adjacency matrix of $H$ differs from that of $K_{a,b}$ in $o(m)$ entries. Then $\lambda_1(A(H)-A(K_{a,b}))=o(\sqrt{m})$. By Lemma~\ref{lem2.1},
\[
 \lambda_2(A(H))\leq\lambda_1(A(H)-A(K_{a,b}))+\lambda_2(A(K_{a,b}))=o(\sqrt{m}).
\]
Together with $\lambda_1(A(H))=\rho (H)=\Theta(\sqrt{m})$, one has
\begin{align}\label{eq3.010}
\lambda_1(A(H))-\lambda_2(A(H))=\Theta(\sqrt{m}).
\end{align}
On the other hand, %since $t=0,$ one has $V(G)=V(H).$
Lemmas~\ref{lem:book-core} and \ref{lem:one-edge-leakage} give $e(H')-h=O_\varepsilon(1),$ and so the adjacency matrix of $H'$ differs from that of $H$ in $O_\varepsilon(1)$ entries. By Lemma~\ref{lem2.06} and \eqref{eq3.010}, the unit
Perron vectors of $H$ and $H'$ differ in Euclidean norm by
$O(m^{-1/2})$. For the Perron vector ${\bf z}$ of $H$, Lemma~\ref{lem2.3} and
\eqref{eq:high-cross-degree} give
\[
 \text{$z_u\ge(1-o(1))\sqrt{\frac{b}{2m}}$ for $u\in A,$
 $z_v\ge(1-o(1))\sqrt{\frac{a}{2m}}$ for $v\in B$.}
\]
Since the main terms $\sqrt{\frac{b}{2m}}$ and $\sqrt{\frac{a}{2m}}$ have order $m^{-1/4}$, the
$O(m^{-1/2})$ perturbation from ${\bf z}$ to ${\bf y}$ is negligible, and so
$$
y_u\ge(1-o(1))\sqrt{\frac{b}{2m}}\quad(u\in A),
 \qquad
 y_v\ge(1-o(1))\sqrt{\frac{a}{2m}}\quad(v\in B).
$$
This completes the proof of this lemma.
\end{proof}

Let
\[
 r_A=e(H'[A]),\qquad r_B=e(H'[B]),
\]
and define
\begin{equation}\notag
 c:=ab-e_{H'}(A,B),\qquad
 s_A:=\lfloor a/2\rfloor-r_A,\qquad
 s_B:=\lfloor b/2\rfloor-r_B.
\end{equation}
Extend the matchings in $H'[A]$ and $H'[B]$ to maximum matchings and add every
missing edge between $A$ and $B$ in $H'$. Denote the resulting graph by $\widehat G$.

\begin{lemma}%[Completion inequality]
\label{lem:completion-inequality}
It holds that
\begin{align}
 \gamma(H')\ge\gamma(\widehat G)+(1-o(1))c
+\bigl(2\sqrt{b/a}-1-o(1)\bigr)s_A
  +\bigl(2\sqrt{a/b}-1-o(1)\bigr)s_B.
\label{eq:completion-inequality}
\end{align}
\end{lemma}

\begin{proof}
Let $\mathcal M$ be the set of edges added to $H'$ to obtain
$\widehat G$. Put
\[
 q:=2\sum_{uv\in\mathcal M}y_uy_v.
\]
The Rayleigh quotient gives $\rho(\widehat G)\ge\rho(H')+q$, and so %Since the function $f(x)= x^2-x$ is increasing in $[\frac{1}{2},+\infty)$,
\begin{align}\notag
 \gamma(H')-\gamma(\widehat G)
 &=-|\mathcal M|+
 \bigl(\rho(\widehat G)^2-\rho(\widehat G)\bigr)
 -\bigl(\rho(H')^2-\rho(H')\bigr)\\\notag
 &\ge-|\mathcal M|+(2\rho(H')-1)q\\\label{eq3.011}
 &= -c-s_A-s_B+(2\rho(H')-1)q.
\end{align}

Let $\mathcal M_{AB},\,\mathcal M_{A}$ and $\mathcal M_{B}$ be the set of edges in $\mathcal M$ contained in $\widehat G[A,B],\,\widehat G[A]$ and $\widehat G[B]$, respectively. Then by \eqref{eq:lower-coordinates-G},
\begin{align}\notag
q&=2\sum_{uv\in\mathcal M_{AB}}y_uy_v+2\sum_{uv\in\mathcal M_{A}}y_uy_v
+2\sum_{uv\in\mathcal M_{B}}y_uy_v \\ \notag
&\geq 2c(1-o(1))\sqrt{\frac{b}{2m}}\sqrt{\frac{a}{2m}}+2s_A(1-o(1))\frac{b}{2m}
+2s_B(1-o(1))\frac{a}{2m}\\ \notag
&=c(1-o(1))\frac{\sqrt{ab}}{m}+s_A(1-o(1))\frac{b}{m}
+s_B(1-o(1))\frac{a}{m}.
\end{align}
Note that $\rho(H')=(1+o(1))\rho(G)=(1+o(1))\sqrt m$, $ab=(1+o(1))m$. Together with \eqref{eq3.011} and \eqref{eq3.012}, one has
\begin{align}\notag
 \gamma(H')-\gamma(\widehat G)&\geq-c-s_A-s_B+(2-o(1))c+\left(2\sqrt{\frac{b}{a}}-o(1)\right)s_A
 +\left(2\sqrt{\frac{a}{b}}-o(1)\right)s_B \\ \notag
 &=(1-o(1))c+\left(2\sqrt{\frac{b}{a}}-1-o(1)\right)s_A
 +\left(2\sqrt{\frac{a}{b}}-1-o(1)\right)s_B,
\end{align}
as desired.
\end{proof}
\begin{lemma}\label{lem3.001}
For $0\le p\le\lfloor a/2\rfloor$ and $0\le q\le\lfloor b/2\rfloor$ with $\alpha=\frac a2-p=O(1)$ and $\zeta=\frac b2-q=O(1)$, one has
\begin{equation}
 \gamma(R(a,b;p,q))=
 \frac{(\sqrt a-\sqrt b)^2}{2}
 +(2\sqrt{b/a}-1)\alpha+(2\sqrt{a/b}-1)\zeta+O((ab)^{-1/2}).
\label{eq:R-defect-expansion}
\end{equation}
\end{lemma}
\begin{proof}
Note that
\begin{align}\label{eq3.012}
e(R(a,b;p,q))=ab+p+q=ab+\frac{a+b}{2}-\alpha-\zeta,
\end{align}
and by Lemma~\ref{lem:R-equation}, $\rho':=\rho(R(a,b;p,q))$ is the largest root of
$x^2(x-1)^2=(ax-2\alpha)(bx-2\zeta).$

Then
\begin{align}\label{eq3.013}
\rho'(\rho'-1)=\sqrt{(a\rho'-2\alpha)(b\rho'-2\zeta)}
=\sqrt{ab}\rho'-\frac{a\zeta+b\alpha}{\sqrt{ab}}+O((ab)^{-1/2}).
\end{align}
Since $\rho'=(1+o(1))\sqrt{ab},$ \eqref{eq3.013} gives
$$
\rho'=\sqrt{ab}+1-\frac{a\zeta+b\alpha}{(\sqrt{ab}+1)\sqrt{ab}}+O((ab)^{-1}).
$$
Together this with \eqref{eq3.012}, one has
\begin{align*}
\gamma(R(a,b;p,q))&=e(R(a,b;p,q))-\rho'(\rho'-1)\\
&=ab+\frac{a+b}{2}-\alpha-\zeta-\left(ab+\sqrt{ab}
-\frac{(2\sqrt{ab}+1)(a\zeta+b\alpha)}{(\sqrt{ab}+1)\sqrt{ab}}+O((ab)^{-1/2})\right)\\
&=\frac{(\sqrt a-\sqrt b)^2}{2}
 +(2\sqrt{b/a}-1)\alpha+(2\sqrt{a/b}-1)\zeta+O((ab)^{-1/2}),
\end{align*}
as desired.
\end{proof}

\begin{lemma}\label{lem3.011}
It holds that $\frac ba=1+o(1).$
\end{lemma}
\begin{proof}
By the definition of $\widehat G$, one sees that $\widehat G\cong R(a,b;p,q)$ with $\alpha=\frac a2-p\in\{0,1/2\}$ and $\zeta=\frac b2-q\in\{0,1/2\}$. Denote by $r:=b/a$, by Lemma~\ref{lem:core-dichotomy}(ii), $1\leq r<100.$ If $1\le r<4$, then both $2\sqrt{b/a}-1$ and $2\sqrt{a/b}-1$ are nonnegative. By \eqref{eq:completion-inequality} and \eqref{eq:R-defect-expansion},
\begin{align}\label{eq3.014}
 \gamma(H')\geq
 \frac{(\sqrt a-\sqrt b)^2}{2}+O((ab)^{-1/2})=\frac a2(1+r-2\sqrt r)+O((ab)^{-1/2}).
\end{align}
Since $a=\Theta(\sqrt m)$ is sufficiently large, and by \eqref{eq:defect-one} and \eqref{eq3.026}, $\gamma(H')\le1$, from \eqref{eq3.014} one gets $r=1+o(1)$.

If $r\geq4$, then $2\sqrt{b/a}-1$ is positive, but $2\sqrt{a/b}-1$ is nonpositive. By \eqref{eq:completion-inequality} and \eqref{eq:R-defect-expansion},
\begin{align}\notag
 \gamma(H')&\geq
 \frac{(\sqrt a-\sqrt b)^2}{2}+(2\sqrt{\frac{a}{b}}-1)\zeta+O((ab)^{-1/2})+
 \bigl(2\sqrt{\frac{a}{b}}-1-o(1)\bigr)s_B\\ \notag
 &\geq\frac a2(1+r-2\sqrt r)+\sqrt{\frac{1}{r}}-\frac{1}{2}
 +a\left(\sqrt{r}-\frac{r}{2}+o(1)\right) \\ \notag
 &\geq \frac{a-1}{2},
\end{align}
a contradiction to $a=\Theta(\sqrt m)$ being sufficiently and $\gamma(H')\le1$. Therefore, $r=1+o(1)$.
\end{proof}
\begin{lemma}%[Discrete defect lemma]
\label{lem:discrete-defect}
If $t=0$, then $c+s_A+s_B\le1,$ and the only possibilities for $G$ are:
\begin{wst}%[label=\textup{(\roman*)}]
 \item[{\rm (i)}] $G=E_{a,b}$;
 \item[{\rm (ii)}] $G=O_{e,o}$;
 \item[{\rm (iii)}] $G$ is obtained from $E_{a,b}$ by deleting one edge inside a part;
 \item[{\rm (iv)}] $G$ is obtained from $K_{a,b}$ by adding a matching of size $\frac{a-1}{2}$ in the $a$-part and a matching of size $\frac{b-1}{2}$ in the $b$-part, where both $a$ and $b$ are odd;
 \item[{\rm (v)}] $G$ is obtained from $E_{a,b}$ by deleting one edge between two parts.
\end{wst}
\end{lemma}

\begin{proof}
Since $t=0,$ $G=H'.$ By Lemma~\ref{lem3.011}, all three coefficients $1-o(1),\,2\sqrt{b/a}-1-o(1)$ and $2\sqrt{a/b}-1-o(1)$ on the right-hand side of \eqref{eq:completion-inequality} equal $1+o(1)$. By the construction of $\widehat G,$ one sees $\widehat G$ is $W_5$-free, and so $\gamma(\widehat G)\ge0$.
Combining \eqref{eq:completion-inequality} and $\gamma(G)\le1$, one has
$c+s_A+s_B\leq 1$.

Note that $\alpha=\frac a2-p\in\{0,1/2\}$ and $\zeta=\frac b2-q\in\{0,1/2\}.$ If both $\alpha$ and $\zeta$ are $0,$ then both $a$ and $b$ are even, and so $c+s_A+s_B\leq 1$ gives one of (i), (iii) and (v). If one of $\alpha$ and $\zeta$ is $0,$ and the other is $\frac{1}{2}$, then formulas \eqref{eq:completion-inequality}, \eqref{eq:R-defect-expansion} and Lemma~\ref{lem3.011} together with $\gamma(G)\le1$ give $c+s_A+s_B= 0$, this gives (ii). If both $\alpha$ and $\zeta$ are $\frac{1}{2}$, then formulas \eqref{eq:completion-inequality}, \eqref{eq:R-defect-expansion} and Lemma~\ref{lem3.011} together with $\gamma(G)\le1$ also give $c+s_A+s_B= 0$, this gives (iv).
\end{proof}

%\section{Exact elimination of the two residual candidates}

%We first eliminate the odd--odd pattern.

For odd $a, b$, let $T_{a,b}$ be obtained from $K_{a,b}$ by inserting a maximum matching in each part.

\begin{lemma}
\label{lem:odd-odd-elimination}
$G$ is not isomorphic to $T_{a,b}$ for any odd $a,b\geq 3$.
\end{lemma}

\begin{proof}
Suppose to the contrary that $G$ is isomorphic to $T_{a,b}$ for some odd $a,b$. Without loss of generality, we may assume $a\leq b.$ Write
\[
 c=\frac{a+b}{2},\qquad d=\frac{b-a}{2}.
\]
Since $e(T_{a,b})=ab+\frac{a-1}{2}+\frac{b-1}{2}=c^2-d^2+c-1$ is odd, $d$ is even.

If $d=0$, then $a=b$, and the graph $E_{a-1,a+1}$ has the same number of edges as that of $T_{a,a}$. By Proposition~\ref{prop:candidate-radii}(i), $\rho(E_{a-1,a+1})=1+\sqrt{a^2-1}$. On the other hand, by Lemma~\ref{lem:R-equation}, $\rho(T_{a,a})$ is the largest zero of $x^2(x-1)^2-(ax-1)^2=0$. As
$$
x^2(x-1)^2-(ax-1)^2=(x^2-(a+1)x+1)(x^2+(a-1)x-1),
$$
$\rho(T_{a,a})=\frac{a+1+\sqrt{(a+1)^2-4}}{2}$, the largest zero of $x^2-(a+1)x+1=0.$ Substitution of $x=1+\sqrt{a^2-1}$ into $x^2-(a+1)x+1$ gives $(a-1) \left(a- \sqrt{a^2-1}\right)>0$. Together with $1+\sqrt{a^2-1}>\frac{a+1}{2}$ gives
$\rho(E_{a-1,a+1})>\rho(T_{a,a})=\rho(G)$, a contradiction to the choice of $G$.

If $d\ge2$, then $m=e(T_{a,b})=c^2-d^2+c-1$. By \eqref{eq:book-defect} and Lemma~\ref{lem:R-equation}, $\beta_m$ is the largest zero of $x(x-1)=m-1,$ and $\rho(T_{a,b})$ is the largest zero of $x^2(x-1)^2=(ax-1)(bx-1).$ Substitution of $x=\beta_m$ into $f(x)=x^2(x-1)^2-(ax-1)(bx-1)$ gives
\begin{align}\notag
f(\beta_m)&=\beta_m^2(\beta_m-1)^2-(a\beta_m-1)(b\beta_m-1)\\ \notag
&=\beta_m^2(\beta_m-1)^2-((c-d)\beta_m-1)((c+d)\beta_m-1)\\ \label{eq3.0020}
&=\beta_m^2\left((\beta_m-1)^2-c^2+d^2\right)+2c\beta_m-1.
\end{align}
Note that
\begin{align}\label{eq3.0021}
\beta_m^2-\beta_m=m-1=e(T_{a,b})-1=c^2-d^2+c-2,
\end{align}
one gets
$
(\beta_m-1)^2-c^2+d^2=-\beta_m+c-1.
$
Now \eqref{eq3.0020} gives
\begin{align}\label{eq3.0023}
f(\beta_m)=\beta_m^2(-\beta_m+c-1)+2c\beta_m-1=\beta_m((\beta_m+2)(c-\beta_m+1)-2)-1.
\end{align}
On the other hand, by \eqref{eq3.0021}, one has
\begin{align}\notag
(c-\beta_m+1)(c+\beta_m)=d^2+2\geq 6,
\end{align}
Substitution of this into \eqref{eq3.0023} gives $f(\beta_m)>0,$ and so
$\rho(T_{a,b})<\beta_m\leq \rho(G)$, a contradiction to the choice of $G$.
\end{proof}

Let $C_{a,b}$ be obtained from $E_{a,b}$ by deleting one edge between two parts.

\begin{lemma}
\label{lem:cross-deletion}
Let $4\le a\le b$ be even. Then
\(
 \rho(C_{a,b})<\rho(D_{a,b}).
\)
%For $a=b=2$ the two graphs are isomorphic.
\end{lemma}

\begin{proof}
For $C_{a,b}$, assume it is obtained from $E_{a,b}$ by deleting the edge $u_1v_1$, where $u_1,v_1$ are in the $a$-part and $b$-part of $E_{a,b}$, respectively. Let $u_2$ be the neighbor of $u_1$ in the $a$-part, and let $v_2$ be the neighbor of $v_1$ in the $b$-part. Partition $V(C_{a,b})$ into $\{u_1\}$, $\{u_2\}$, the set of remaining $a-2$ vertices in the $a$-part, $\{v_1\}$, $\{v_2\}$, and the set of remaining $b-2$ vertices in the $b$-part. The quotient matrix of $A(C_{a,b})$ corresponding to this equitable partition is
\[
 Q_{a,b}=\begin{pmatrix}
 0&1&0&0&1&b-2\\
 1&0&0&1&1&b-2\\
 0&0&1&1&1&b-2\\
 0&1&a-2&0&1&0\\
 1&1&a-2&1&0&0\\
 1&1&a-2&0&0&1
 \end{pmatrix}.
\]
By a direct calculation, the characteristic polynomial of $Q_{a,b}$ is
\begin{equation}\notag%\label{eq:cross-polynomial}
 q_{a,b}(x):=x^6-2x^5-abx^4+6x^3+(3ab-a-b-5)x^2-(a+b)x+4-ab;
\end{equation}
and by Lemma~\ref{lem2.05}, $\rho(C_{a,b})$ is the largest zero of $q_{a,b}(x)$.

Recall that $\rho(D_{a,b})$ is the largest zero of $g_{a,b}(x)$ (see \eqref{eq:D-cubic}).
Reducing $q_{a,b}(x)$ modulo $g_{a,b}(x)$ gives
\begin{equation}
 R_{a,b}(x)=(2ab-5a-b+4)x^2+(-2a^2b+4ab+3a-b-4)x+4a^2-ab-8a+4.
\label{eq:remainder}
\end{equation}
Writing $b=a+d$, we have
\begin{equation}
 R_{a,a+d}(x)=R_{a,a}(x)+dK_a(x),
\label{eq:remainder-split}
\end{equation}
where
\(
 K_a(x)=(2a-1)x^2-(2a^2-4a+1)x-a.
\)
The polynomial $K_a(x)$ is positive and increasing for $x\ge a$, while
$R_{a,a}(x)$ is increasing for $x\ge a$.

If $a=b$, then $\rho(C_{a,a})$ is the largest zero of
\(
 q_a^{\ast}(x)=x^3-ax^2-(a+1)x+a+2,
\)
$\rho(D_{a,a})$ is the largest zero of
\(
 g_{a,a}(x)=x^3-2x^2+(1-a^2)x+2a.
\)
Then direct reduction gives
\begin{equation}\label{eq3.0016}
 g_{a,a}(\rho(C_{a,a}))=(a-2)\bigl(\rho^2(C_{a,a})-(a+1)\rho(C_{a,a})+1\bigr).
\end{equation}
Let $\tau$ be the largest zero of $f_a(x)=x^2-(a+1)x+1,$ then $q_a^{\ast}(\tau)=a+1-\tau>0$, and so $\rho(C_{a,a})<\tau.$ Now \eqref{eq3.0016} gives $g_{a,a}(\rho(C_{a,a}))<0,$ and so $\rho(D_{a,a})>\rho(C_{a,a})$, which implies
$R_{a,a}(\rho(D_{a,a}))>0$.

Note that $a\leq b$, the graph $D_{a,a}$ is a subgraph of $D_{a,b}$, so
$\rho(D_{a,b})\ge\rho(D_{a,a})\ge a$. Equations
\eqref{eq:remainder}--\eqref{eq:remainder-split} now give
\begin{equation}\label{eq3.0017}
q_{a,b}(\rho(D_{a,b}))=R_{a,b}(\rho(D_{a,b}))>0.
\end{equation}
Note that $C_{a,b}$ has only one eigenvalue greater than $2$, and $\rho(D_{a,b})>2$, \eqref{eq3.0017} implies $\rho(D_{a,b})>\rho(C_{a,b})$.
\end{proof}

By Lemma~\ref{lem:R-equation}, the spectral radius of the graph obtained by deleting an edge in the $a$-part is the largest zero of $g_{b,a}(x)=x^3-2x^2+(1-ab)x+2b$. Since $a\le b$, the largest zero of $g_{b,a}(x)$ is smaller than that of $g_{a,b}(x)$. Combining this observation with Lemmas~\ref{lem:discrete-defect}--\ref{lem:cross-deletion} proves the
following results.

\begin{proposition}
\label{prop:balanced-reduction}
If $t=0$, then
\(
 G\in\cE_m\cup\cO_m\cup\cD_m.
\)
\end{proposition}

%\section{The pendant alternative} We now consider the case $t=1$.

\begin{lemma}
\label{lem:pendant-rigidity}
If $t=1$, then
\(
 m=n^2+n+1
\)
for an even integer $n$, and $G\cong H_n$.
\end{lemma}

\begin{proof}
By Lemma~\ref{lem:one-edge-leakage}, the unique edge outside $H'$ is a pendant edge. Let
\(
 \theta=\bigl(\rho^2-\rho\bigr)
 -\bigl(\rho(H')^2-\rho(H')\bigr),
\)
\eqref{eq:defect-gain-small} gives $\theta=O(m^{-1/2})$. Since
$\gamma(G)=\gamma(H')+1-\theta\le1$, we have
$0\le\gamma(H')\le\theta=o(1)$.

Now by \eqref{eq:completion-inequality} and Lemma~\ref{lem3.011}, $c+s_A+s_B=1;$ and by \eqref{eq:R-defect-expansion}, $\alpha=\zeta=0.$ Hence $H'=E_{a,b}$ for even $a\le b$, and so
\begin{equation}
 \rho(H')=1+\sqrt{ab},\qquad
 \gamma(H')=\frac{a+b}{2}-\sqrt{ab}
 =\frac{(\sqrt b-\sqrt a)^2}{2}.
\label{eq:E-defect}
\end{equation}

%We compare this defect with the gain from a pendant edge.  The spectrum of $E_{a,b}$ consists of $1+\sqrt{ab}$, $1-\sqrt{ab}$, and eigenvalues in $\{-1,1\}$.  If the pendant edge is attached at a vertex $v$ in the $a$-part, the squared principal spectral weight at $v$ is $1/(2a)$; in the $b$-part it is $1/(2b)$.
Denote by $\rho=\rho(H')+\zeta$. Let $z$ be the pendant vertex of $G$, and let $v$ be the unique neighbor of $z.$ Now
$$
A(G)=\left(
       \begin{array}{cc}
         A(H') & e_v \\
         e_v^T & 0 \\
       \end{array}
     \right),
$$
where $e_v$ is a $|V'|$-dimension vector with $(e_v)_v=1$ and $(e_v)_u=0$ for $u\neq v.$ Now
$$
A(G){\bf x}=\left(
       \begin{array}{cc}
         A(H') & e_v \\
         e_v^T & 0 \\
       \end{array}
     \right)\left(
              \begin{array}{c}
                x_{V'} \\
                x_z \\
              \end{array}
            \right)=\rho\left(
              \begin{array}{c}
                x_{V'} \\
                x_z \\
              \end{array}
            \right)
$$
gives
$$
A(H')x_{V'}+x_ze_v=\rho x_{V'},\qquad e_v^Tx_{V'}=\rho x_z,
$$
and so
\begin{align}\label{eq3.0030}
 \rho=e_v^T(\rho I-A(H'))^{-1}e_v=\bigl((\rho I-A(H'))^{-1}\bigr)_{vv}.
\end{align}

Let ${\bf q}_1,{\bf q}_2,\ldots,{\bf q}_{|V'|}$ be the standard orthogonal eigenvectors of $A(H')$ corresponding to the eigenvalues $\rho(H'):=\mu_1>\mu_2\geq\cdots\geq\mu_{|V'|}.$ The spectral decomposition theory gives
\begin{align}\label{eq3.0031}
\bigl((\rho I-A(H'))^{-1}\bigr)_{vv}=\sum_{i=1}^{|V'|}\frac{({\bf q}_i)_v^2}{\rho-\mu_i}.
\end{align}
Since $H'=E_{a,b}$, the entries of ${\bf q}_1$ corresponding to the vertices in the same part of $H'$ are the same. Together with $A(H'){\bf q}_1=\mu_1{\bf q}_1$ and $||{\bf q}_1||_2=1$, if $v$ is in the $a$-part (resp. $b$-part), then $({\bf q}_1)_v^2=\frac{1}{2a}$ (resp. $({\bf q}_1)_v^2=\frac{1}{2b}$). Now $a\leq b$ gives $({\bf q}_1)_v^2\leq\frac{1}{2a}.$ On the other hand, Lemma~\ref{lem:R-equation} implies $\mu_2\leq 1.$ Together with \eqref{eq3.0030} and \eqref{eq3.0031}, one has
$$
\rho=\sum_{i=1}^{|V'|}\frac{({\bf q}_i)_v^2}{\rho-\mu_i}
\leq\frac{1/(2a)}{\rho-\mu_1}+\frac{\sum_{i=2}^{|V'|}{({\bf q}_i)_v^2}}{\rho-1}
\leq\frac{1/(2a)}{\zeta}+\frac{1}{\rho-1},
$$
which gives
\[
 \zeta\le
 \frac{1/(2a)}{\rho-1/(\rho-1)}.
\]
Since $\rho(H')=1+\sqrt{ab}$ and $\rho=\rho(H')+\zeta$, one has $\rho=(1+o(1))\sqrt{ab},$ this yields
\begin{equation}
 \gamma(H')\le\theta=\bigl(\rho^2-\rho\bigr)
 -\bigl(\rho(H')^2-\rho(H')\bigr)\le\frac{1+o(1)}a.
\label{eq:pendant-gain}
\end{equation}

Recall that $e(H')=m-1$ is even, and both $a$ and $b$ are even. Writing $a=2r$ and $b=2s$, by $e(H')=ab+\frac{a+b}{2}=ab+r+s,$ one sees $s-r$ is even. Thus either $a=b$ or
$b-a\ge4$. In the latter case, \eqref{eq:E-defect} gives
\[
 \gamma(H')=\frac{(b-a)^2}{2(\sqrt a+\sqrt b)^2}
 \ge\frac{2-o(1)}a,
\]
contradicting \eqref{eq:pendant-gain}. Therefore $a=b:=n$, where $n$ is
even.  Hence $H=E_n$, $m=n^2+n+1$, and $G=H_n$.
\end{proof}

%The pendant graph is genuinely better than every core-contained candidate at its size.

\begin{lemma}
\label{lem:special-dominance}
Let $n$ be sufficiently large and even, and put $m=n^2+n+1$.  Every
graph in
\(
 \{B_m\}\cup\cE_m\cup\cO_m\cup\cD_m
\)
has spectral radius at most $n+1$, whereas
$\rho(H_n)>n+1$.
\end{lemma}

\begin{proof}
We have $\rho(B_m)=\beta_m=\frac{1+\sqrt{4m-3}}2=n+1$. Also $H_n$ contains $E_n$ as a proper subgraph, so Perron--Frobenius gives $\rho(H_n)>\rho(E_n)=n+1$.

It remains to bound the spectral radii of the graphs in $\cE_m\cup\cO_m\cup\cD_m.$ We give one calculation covering all of them. Use the notation \eqref{eq:alpha-eta},
%and put
%\[
% c=\frac{a+b}{2},\qquad d=\frac{b-a}{2},\qquad
% A=\alpha+\zeta,\qquad w=c-n.
%\]
and let
\[
 L(x)=x^2(x-1)^2-(ax-2\alpha)(bx-2\zeta),
\]
where $ab+\frac{a+b}{2}-(\alpha+\zeta)=m.$
Then
\begin{align*}
 L(n+1)&=(n+1)^2n^2-(a(n+1)-2\alpha)(b(n+1)-2\zeta)\\
 &=(n+1)^2n^2-ab(n+1)^2+2(a\zeta+b\alpha)(n+1)-4\alpha\zeta.
\end{align*}
We consider the following three cases with respect to the values of $\alpha$ and $\zeta$.

%The edge equation is equivalent to
%\begin{equation}
% w(2n+w+1)=d^2+A+1.
%\label{eq:special-edge-relation}
%\end{equation}

%Substitution of \eqref{eq:special-edge-relation} gives
%\begin{equation}
%\begin{split}
% L(n+1)={}&(n+1)^2(w-1)
% +A(n+1)(n+2w-1)\\
% &+2(n+1)d(\alpha-\zeta)-4\alpha\zeta.
%\end{split}
%\label{eq:special-sign}
%\end{equation}
\textbf{Case 1: $\alpha=\zeta=0$.} In this case, $\rho(E_{a,b})$ is the largest zero of $L(x).$ On the other hand,
$$
L(n+1)=(n+1)^2n^2-ab(n+1)^2=(n+1)^2(n^2-ab).
$$
Note that $ab+\frac{a+b}{2}=m=n^2+n+1.$ If $ab\geq n^2+1,$ then $\frac{a+b}{2}\leq n,$ and so $\sqrt{ab}\leq \frac{a+b}{2}\leq n,$ this contradicts to $ab\geq n^2+1.$ Hence $ab\leq n^2,$ and so $L(n+1)\geq 0.$ This gives $\rho(E_{a,b})\leq n+1.$

\textbf{Case 2: $\alpha=0,\,\zeta=\frac{1}{2}$.} In this case, $a$ is even, $b$ is odd, and $\rho(O_{a,b})$ is the largest zero of $L(x).$ On the other hand,
$$
L(n+1)=(n+1)^2n^2-ab(n+1)^2+a(n+1)=(n+1)(n^2(n+1)-ab(n+1)+a).
$$
Note that $ab+\frac{a+b-1}{2}=m=n^2+n+1.$ If $ab>n^2,$ then since both $a$ and $n$ are even, $ab\geq n^2+2.$ Now $\frac{a+b-1}{2}\leq n-1,$ and so $\sqrt{ab}\leq \frac{a+b}{2}\leq n-\frac{1}{2},$ this contradicts to $ab\geq n^2+2.$ Hence $ab\leq n^2,$ and so $L(n+1)\geq 0.$ This gives $\rho(O_{a,b})\leq n+1.$

\textbf{Case 3: $\alpha=0,\,\zeta=1$.} In this case, both $a$ and $b$ are even, and $\rho(D_{a,b})$ is the largest zero of $L(x).$ On the other hand,
$$
L(n+1)=(n+1)^2n^2-ab(n+1)^2+2a(n+1)=(n+1)(n^2(n+1)-ab(n+1)+2a).
$$
Note that $ab+\frac{a+b-2}{2}=m=n^2+n+1.$ If $ab>n^2,$ then since all of $a$, $b$ and $n$ are even, $ab\geq n^2+4.$ Now $\frac{a+b-2}{2}\leq n-3,$ and so $\sqrt{ab}\leq \frac{a+b}{2}\leq n-2,$ this contradicts to $ab\geq n^2+4.$ Hence $ab\leq n^2,$ and so $L(n+1)\geq 0.$ This gives $\rho(D_{a,b})\leq n+1.$

This completes the proof of this lemma.
\end{proof}

%\section{Proof of the main theorem}

\begin{proof}[\bf Proof of Theorem~\ref{thm:main}]
Let $G$ be a $W_5$-free graph of size $m$ with no isolated vertex having the largest spectral radius, where $m$ is odd and sufficiently. By the Perron-Frobenius theorem, the graph $G$ is connected.

Apply Lemma~\ref{lem:book-core}, $G$ contains a connected $\varepsilon$-core $H$. If the alternative (i) of Lemma~\ref{lem:core-dichotomy} occurs for $H$, then
$\rho(H)=\beta_{e(H)}$. The final assertion of Lemma~\ref{lem:book-core} gives $H=G$, and
Lemma~\ref{lem:core-dichotomy} yields $G=B_m$.

We may therefore assume that the alternative (ii) of Lemma~\ref{lem:core-dichotomy} occurs for $H$. Let $t=e(G)-e(G[V(H)])$. By Lemma~\ref{lem:one-edge-leakage}, $t\le1$. If $t=0,$ then Proposition~\ref{prop:balanced-reduction} shows $G\in\cE_m\cup\cO_m\cup\cD_m.$ If $t=1$, then Lemma~\ref{lem:pendant-rigidity} shows $m=n^2+n+1$ for an even integer $n$, and $G= H_n$.

Finally by Lemma~\ref{lem:special-dominance}, if $m=n^2+n+1$ for an even integer $n$, then $G= H_n$; otherwise, $G\in\{B_m\}\cup\cE_m\cup\cO_m\cup\cD_m.$ This completes the proof of this theorem.
\end{proof}

\section{Concluding remarks}

Theorem~\ref{thm:main} presents the graphs with maximum spectral radius among all $W_5$-free graphs of size $m$ with $m$ being odd and sufficiently large. Note that the candidate $B_m$ are well-defined for all odd $m,$ it is interesting to try to compare the spectral radius of $B_m$ with those of the graphs in $\cE_m\cup\cO_m\cup\cD_m.$

\begin{proposition}\label{prop:book-tests}
Let $m$ be sufficiently large and odd, and put $\theta=\beta_m$. Then
\begin{align*}
 \rho(O_{e,o})>\theta
 &\iff (m-1-eo)\theta+e-m+1<0,\\
 \rho(D_{a,b})>\theta
 &\iff (m-1-ab)\theta+2a-m+1<0.
\end{align*}
Equality is characterized by replacing both strict inequalities with
equalities. Also,
\[
 \rho(E_{a,b})>\theta\iff 1+\sqrt{ab}>\theta.
\]
\end{proposition}

\begin{proof}
Use $\theta^2-\theta=m-1$ to reduce
$f_{e,o}(\theta)$ and $g_{a,b}(\theta)$.  This gives the two displayed
linear expressions.  By Lemma~\ref{lem:R-equation}, each polynomial has
only one root greater than $1$, and it is increasing through that root.
The last assertion follows from Proposition~\ref{prop:candidate-radii}.
\end{proof}

All three families $\cE_m$, $\cO_m$, $\cD_m$ are necessary. For example, for even
$s\to\infty$, the graphs $E_{s,s+2}$, $O_{s,s-1}$, and $D_{s,s}$ have
odd sizes $s^2+3s+1$, $s^2-1$, and $s^2+s-1$, respectively. Direct
substitution in Proposition~\ref{prop:book-tests} shows that each has
spectral radius larger than $\rho(B_m)$ with $m=s^2+3s+1$, $s^2-1$, and $s^2+s-1$, respectively. %Thus the exact winner genuinely depends on the arithmetic of $m$.
\section*{Statements and Declarations}

\textbf{Competing interests}
The authors declare that they have no competing interests.

\textbf{Data availability}
No data were generated or analysed during the current study.

\textbf{Funding}
Shuchao Li was supported by the National Natural Science Foundation of China (Grant Nos. 12571365, 12171190).

%\noindent\textbf{Author contributions}
%All authors contributed to the conception and development of the results, the writing of the manuscript, and the review of the final version.

\textbf{Use of AI tools}
During the preparation of this manuscript, the authors used Doubao for language polishing and structural organization. The authors reviewed and edited all content and take full responsibility for the final manuscript.

\end{document}